\documentclass[final,3p,times]{elsarticle}
\usepackage{epsfig}\usepackage{graphicx}
\usepackage{amsmath,amssymb,amsfonts}
\usepackage{amsthm}

\usepackage{booktabs}

\newcommand{\comment}[1]{}

\usepackage[bookmarks,colorlinks,breaklinks]{hyperref}
\hypersetup{
    colorlinks,%
    citecolor=black,%
    filecolor=black,%
    linkcolor=black,%
    urlcolor=black
}

\newtheorem{thm}{Theorem}[section]
\newtheorem{cor}[thm]{Corollary}

\makeatletter
\newcommand{\xleftrightarrow}[2][]{\ext@arrow 3359\leftrightarrowfill@{#1}{#2}}
\newcommand{\xdashrightarrow}[2][]{\ext@arrow 0359\rightarrowfill@@{#1}{#2}}
\newcommand{\xdashleftarrow}[2][]{\ext@arrow 3095\leftarrowfill@@{#1}{#2}}
\newcommand{\xdashleftrightarrow}[2][]{\ext@arrow 3359\leftrightarrowfill@@{#1}{#2}}
\def\rightarrowfill@@{\arrowfill@@\relax\relbar\rightarrow}
\def\leftarrowfill@@{\arrowfill@@\leftarrow\relbar\relax}
\def\leftrightarrowfill@@{\arrowfill@@\leftarrow\relbar\rightarrow}
\def\arrowfill@@#1#2#3#4{%
  $\m@th\thickmuskip0mu\medmuskip\thickmuskip\thinmuskip\thickmuskip
   \relax#4#1
   \xleaders\hbox{$#4#2$}\hfill
   #3$%
}
\makeatother

\begin{document}
\title{Iterated Diffusion Maps for Feature Identification}
\author[rvt]{Tyrus Berry\corref{cor}}
\ead{tberry@gmu.edu}
\author[rvt,rvt2]{John Harlim}
\ead{jharlim@psu.edu}
\cortext[cor]{Corresponding author}
\address[rvt]{Department of Mathematical Sciences, George Mason University, 4400 Exploratory Hall, Fairfax, Virginia  22030, USA}
\address[rvt2]{Department of Meteorology, the Pennsylvania State University, 503 Walker Building, University Park, PA 16802-5013, USA}
\date{\today}

\begin{abstract}
Recently, the theory of diffusion maps was extended to a large class of \emph{local kernels} with exponential decay which were shown to represent various Riemannian geometries on a data set sampled from a manifold embedded in Euclidean space.  Moreover, local kernels were used to represent a diffeomorphism $\mathcal{H}$ between a data set and a feature of interest using an anisotropic kernel function, defined by a covariance matrix based on the local derivatives $D\mathcal{H}$.  In this paper, we generalize the theory of local kernels to represent degenerate mappings where the intrinsic dimension of the data set is higher than the intrinsic dimension of the feature space.  First, we present a rigorous method with asymptotic error bounds for estimating $D\mathcal{H}$ from the training data set and feature values.  We then derive scaling laws for the singular values of the local linear structure of the data, which allows the identification the tangent space and improved estimation of the intrinsic dimension of the manifold and the bandwidth parameter of the diffusion maps algorithm. Using these numerical tools, our approach to feature identification is to iterate the diffusion map with appropriately chosen local kernels that emphasize the features of interest.  We interpret the iterated diffusion map (IDM) as a discrete approximation to an intrinsic geometric flow which smoothly changes the geometry of the data space to emphasize the feature of interest. When the data lies on a manifold which is a product of the feature manifold with an irrelevant manifold, we show that the IDM converges to the quotient manifold which is isometric to the feature manifold, thereby eliminating the irrelevant dimensions. We will also demonstrate empirically that if we apply the IDM to features which are not a quotient of the data manifold, the algorithm identifies an intrinsically lower-dimensional set embedding of the data which better represents the features. 

\end{abstract}


\maketitle

\section{Introduction}\label{section1}

Often, for high-dimensional data and especially for data lying on a nonlinear subspace of Euclidean space, the variables of interest do not lie in the directions of largest variance and this makes them difficult to identify.  In this paper we consider the supervised learning problem, where we have a training data set, along with the values of the features of interest for this training data. Throughout this manuscript we will assume that the training data set consists of data points which are near a manifold $\mathcal{M}$ embedded in an $m$-dimensional Euclidean space; we refer to $\mathcal{M}$ as the `data space' or `data manifold'.  We also assume that we have a set of feature values corresponding to each training data point, and these feature values are assumed to lie near a manifold $\mathcal{N}$ embedded in an $n$-dimensional Euclidean space; we refer to $\mathcal{N}$ as the `feature space' or `feature manifold'.  We assume that the feature manifold is intrinsically lower-dimensional than the data manifold, so intuitively the data manifold contains information which is irrelevant to the feature, and we refer to this information broadly as the `irrelevant variables' or the `irrelevant space'.  In some contexts we will be able to identify the irrelevant space explicitly, for example the data manifold may simply be a product manifold of the feature manifold with an irrelevant manifold.  However, more complex relationships between the data manifold, feature manifold, and irrelevant variables are possible.  We will think of the feature space as arising from a function defined on the data space, and our goal is to represent this function.  However, this will only be possible in some contexts such as the product manifold described above.  More generally, our goal is to find a mapping from the data space to an intrinsically lower-dimensional space which contains all the information of the feature space but is, as far as possible, independent of the irrelevant variables.

Recently a method for formally representing a diffeomorphism between manifolds using discrete data sets sampled from the manifolds was introduced in \cite{LK}.  In particular, a diffeomorphism is represented using a \emph{local kernel} to \emph{pull back} the Riemannian metric from one manifold onto the other.  With respect to the intrinsic geometry of the local kernel, the manifolds are isometric, and the isometry can be represented by a linear map between the eigenfunctions of the respective Laplacian operators. In this paper, we consider the more difficult case when the manifolds are not diffeomorphic, so that one manifold may even be higher dimensional than the other.  This represents the scenario described above, where some of the variables of a high-dimensional data set may be irrelevant to the features of interest.  

The challenge of having irrelevant variables is that it violates the fundamental assumption of differential geometry, namely that it is local. This is because data points which differ only in the irrelevant variables will be far away in the data space and yet have the same feature values. This fundamental issue is independent of the amount of data available and is illustrated in Figure \ref{annulus1}. Namely, if the feature of interest is the radius of an annulus, then points on opposite sides of the annulus are closely related with respect to this feature of interest. Conversely, points which are far away in the feature space may appear relatively close in data space; this can occur when many of the irrelevant variables are very close. Of course, in the limit of large data, points being close in data space implies that they are close in feature values. However, a large number of irrelevant variables can easily overwhelm any finite data set due to the curse-of-dimensionality. Intuitively, the presence of irrelevant variables makes it difficult to determine the true neighbors. 

\begin{figure}[h]
\begin{center}
\includegraphics[trim={6cm 0 4.5cm 0},clip,width=.95\linewidth]{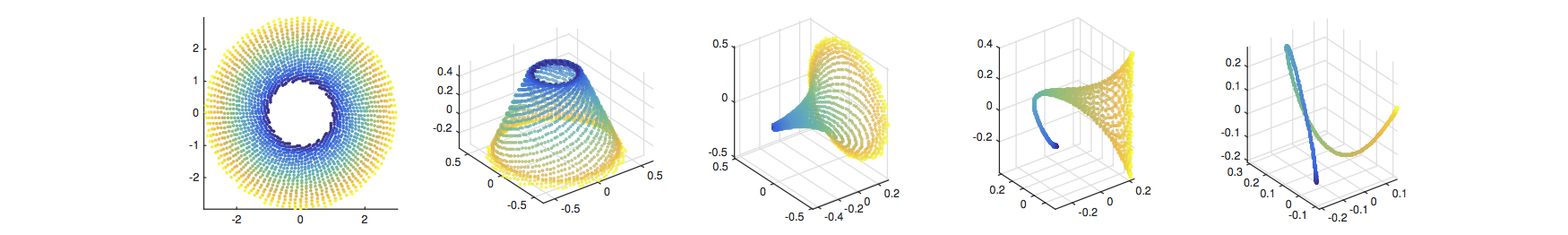}
 \includegraphics[trim={6cm 0 4.5cm 0},clip,width=.95\linewidth]{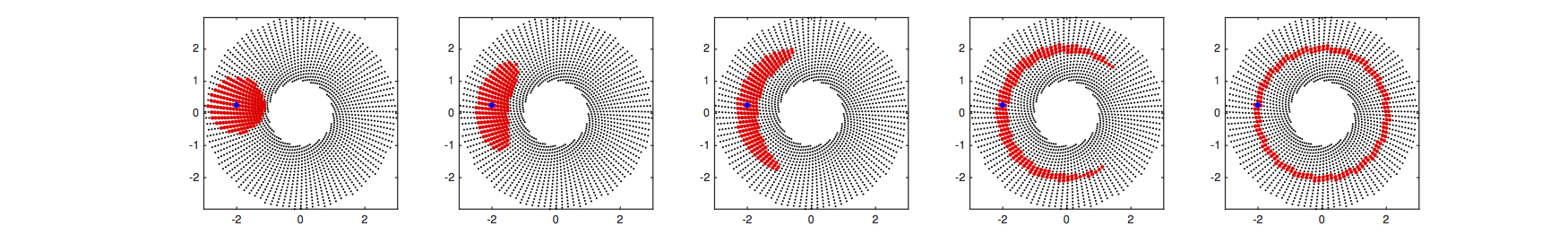}
\caption{\label{annulus1} Top: Original data set colored according to the desired feature (leftmost) followed by four iterations of the diffusion map using a local kernel defined in Section \ref{section4}.  Bottom: Original data set showing the 200 nearest neighbors (red) of the blue data point, where the neighbors are found in the corresponding iterated diffusion map space.  Notice that as the iterated diffusion map biases the geometry towards the desired feature (the radius), the neighbors evolve towards the \emph{true} neighbors with respect to the desired feature.}
\end{center}
\end{figure}

Intuitively our goal is to determine the true neighbors of every point in the data space, meaning the points in the data space which have similar feature values regardless of the values of the irrelevant variables.  The key difficulty is that we need a method which can be extended to new data points, since the goal of representing the map $\mathcal{H}$ is to be able to apply this map to new points in data space.  In order to learn the map $\mathcal{H}$ we will assume that we have a training data set for which the feature values are known.  Of course, for the training data set, we could easily find the true neighbors of the training data points by using the known feature values.  However, finding the neighbors using the feature values cannot be used for determining the true neighbors of a new data point, since the goal is to determine the feature values of the new data point.  Instead, we propose an iterative mapping which smoothly distorts the data space in a way that contracts in the directions of the irrelevant variables and expands in the directions of the features.  Crucially, this iterative mapping can be smoothly extended to new data points for which the feature values are unknown, allowing us to extend the feature map to these new data points.

In this paper, we introduce the Iterated Diffusion Map (IDM) which is an iterative method that smoothly distorts the data space in a way that contracts in the directions of the irrelevant variables and expands in the directions of the features. We illustrate our method for the purposes of intuitive explanation in Figure \ref{annulus1}. In this example, the original data set is an annulus in the plane, but the variable of importance (represented by color in the top images) is the radial component of the annulus, meaning that the angular component is an irrelevant dimension of the manifold.  The initial neighborhood is simply an Euclidean ball in the plane as shown in the bottom row of images.  In the subsequent images we apply the diffusion map multiple times to evolve the data set in a way that biases it towards the desired feature. Notice that as the geometry evolves, the notion of neighborhood evolves.  In the bottom right image, we see that after four iterations of the diffusion map, the notion of neighbor has grown to include any points which have the same radius, independent of their angle. Moreover, after four iterations, we see that points that were initially very close neighbors, namely points that have the same angle but slightly different radii, are no longer neighbors.  So after applying the IDM, the notion of neighbor becomes very sensitive to the feature (radius) and independent of the irrelevant variable (angle). 

As we will explain in Section \ref{productman}, the example in Figure \ref{annulus1} has a particularly nice structure, namely the full data set is a product space of the desired feature with the irrelevant variables.  When this structure is present we will be able to interpret the iterated diffusion map as a discretization of a geometric evolution which contracts the irrelevant variables to zero, thereby reconstructing the quotient map.  When the product space structure is not present, the iterated diffusion map recovers a more complex structure which is not yet theoretically understood.  In Section \ref{examples} we will give several simple examples of both product spaces and non-product spaces that illustrate the current theory and its limitations.  Naturally, if one wishes to understand every feature of a data set, there is no advantage to the iterated diffusion map.  However, often we can identify desirable features in training data sets, and the IDM finds an extendable map to an intrinsically lower dimensional space which better represents the desired features.   

As we will see below, the construction of IDM requires several tools. In Section~\ref{section21}, we will show that iterating the standard diffusion map of \cite{diffusion} has no effect (after the first application of the diffusion map, subsequent applications will approximate the identity map when appropriately scaled).  We will see that the isotropic kernel used in the standard diffusion map yields a canonical isometric embedding of the manifold.  In Section~\ref{section22}, we review how \emph{local kernels}, can change the geometry of the manifold and obtain an isometric embedding with respect to the new geometry.  Local kernels are a broad generalization of the isotropic kernels used in \cite{diffusion} and were shown in \cite{LK} to be capable of representing any geometry on a manifold.  

To construct a local kernel that emphasizes the feature directions, we will need to estimate the derivative of the feature map, $D\mathcal{H}$. In Section~\ref{section3}, we give a rigorous method of estimating $D\mathcal{H}$, including asymptotic error bounds, based on a weighted local linear regression.  Moreover, in Section \ref{svd}, by applying this weighted local linear regression from the manifold to itself we derive scaling laws for the singular values of the local linear structure near a point on the manifold described by the data.  The scaling laws of the singular values yield a robust method of identifying the tangent space of the manifold near a point.  Finally, these scaling laws allow us to devise more robust criteria of determining the intrinsic dimension of manifold, as well as the local bandwidth parameter of the diffusion maps algorithm.  

With the tools of Sections \ref{section2} and \ref{section3}, our goal is to use $D\mathcal{H}$ to construct a new geometry on the data set that emphasizes the feature of interest.  In Section \ref{section4} we will see that naively forming an anisotropic kernel (following \cite{LK}) using a rank deficient matrix $D\mathcal{H}$ will not satisfy the assumptions that define a local kernel.  So, in order to represent the feature map $\mathcal{H}$, we cannot directly apply the local kernels theory of \cite{LK}, which only applies to diffeomorphisms.  Instead, in Section \ref{section4} we introduce the IDM as a discrete approximation of a geometric flow that contracts the irrelevant variables and expands the feature variables on the manifold.  In Section \ref{productman} we show that when the data manifold is the product of the feature manifold with irrelevant variables, this geometric flow will recover the quotient map from the data manifold to the feature manifold.  In Section \ref{examples} we give several numerical examples demonstrating the IDM and we also include the IDM numerical algorithm in \ref{numerics}. We close the paper with a short summary, highlighting the advantages and limitations.

\section{Background}\label{section2}

In this section, we review recent key results that are relevant to the method developed in this paper. First, we remind the readers that, up to a scalar factor, a diffusion map is an isometric embedding of the manifold represented by a data set.  Second, we briefly review the recently developed method for representing diffeomorphism between manifolds \cite{LK}, which we will use as a building block.

\subsection{The Diffusion Map as an Isometric Embedding}\label{section21}

A natural distance that respects the nonlinear structure of the data is the geodesic distance, which can be approximated as the shortest path distance. However, the shortest path distance is very sensitive to small perturbations of a data set. A more robust metric that also respects the nonlinear structure of the data is the diffusion distance which is defined as an average over all paths. This metric can be approximated by the Euclidean distance of the data points in the embedded space constructed by the diffusion maps algorithm \cite{diffusion}.

For a Riemannian manifold $\mathcal{M}$ with associated heat kernel $k(t,x,y)$, we can define the diffusion distance as,
\[ D_t(x,y)^2 = || k(t,x,\cdot)-k(t,y,\cdot)||^2_{L^2(\mathcal{M})} = \int_{\mathcal{M}} (k(t,x,u)-k(t,y,u))^2 dV(u), \]
for $x,y \in \mathcal{M}$ where $dV$ is the volume form on $\mathcal{M}$ associated to the Riemannian metric which corresponds to $k$.  We can write the heat kernel as $k(t,x,y) = \left(e^{t\Delta}\delta_x\right)(y)$ where $\delta_x$ is the Dirac delta function and $\Delta$ is the (negative definite) Laplace-Beltrami operator on $\mathcal{M}$ with eigenfunctions $\Delta \varphi_i = \lambda_i \varphi_i$, where $0 = \lambda_0>\lambda_1>\lambda_2>\ldots$.  Using the Plancherel equality the diffusion distance becomes,
\[ D_t(x,y)^2 = || e^{t\Delta}\delta_x - e^{t\Delta}\delta_y ||^2_{L^2(\mathcal{M})} = \sum_{i=1}^{\infty} \left< e^{t\Delta}\delta_x - e^{t\Delta}\delta_y , \varphi_i \right>^2 = \sum_{i=1}^{\infty} e^{2t\lambda_i}(\varphi_i(x)-\varphi_i(y))^2, \]
where the term $i=0$ is zero since $\varphi_0$ is constant.  Defining the diffusion map by,
\[ \Phi_t(x) = (e^{t\lambda_1}\varphi_1(x),...,e^{t\lambda_M}\varphi_M(x))^\top, \] 
for $M$ sufficiently large, the diffusion distance is well approximated by the Euclidean distance in the diffusion coordinates, $D_t(x,y) \approx ||\Phi_t(x)-\Phi_t(y)||$.  The key to making this idea practical is the algorithm of \cite{diffusion} which uses the data set $\{x_i\}$ sampled from a manifold $\mathcal{M}\subset\mathbb{R}^m$
to construct a sparse graph Laplacian $L$ that approximates the Laplace-Beltrami operator, $\Delta$ on $\mathcal{M}$. 

Of course, the dimension $M$ of the diffusion coordinates will depend on the parameter $t$, which is intuitively a kind of coarsening parameter.  For small $t$ we have, 
\[ \left<e^{t\Delta}\delta_x,\delta_y\right> = (4\pi t)^{-d/2} e^{-d_g(x,y)^2/(4t)}(u_0(x,y)+\mathcal{O}(t)), \] 
where $d_g(x,y)$ is the geodesic distance and $d$ is the intrinsic dimension of $\mathcal{M}$ (see for example \cite{laplacianBook}).  The function $u_0(x,y)$ is the first term in the heat kernel expansion.  In \cite{laplacianBook} the following formula is derived for $u_0(x,y)$,
\[ u_0(x,y) = | d(\textup{exp}_x^{-1})(y) |^{1/2} = | I_{d\times d} + \mathcal{O}(d_g(x,y)^2) |^{1/2} = 1 + \mathcal{O}(d_g(x,y)^d) \]
where $\exp_x$ is the exponential map based at $x$, and the expansion follows from noting that $\textup{exp}_x$ is a smooth map with first derivative equal to the identity at $x$ and second derivative orthogonal to the tangent plane.  Using the expansion of $u_0(x,y)$, for $d_g(x,y)$ sufficiently small, we have the following expansion of the heat kernel,
\begin{align}\label{heatkernel} 
\left<e^{t\Delta}\delta_x,\delta_y\right> = (4\pi t)^{-d/2} e^{-d_g(x,y)^2/(4t)}(1+\mathcal{O}(t,d_g(x,y)^d)), 
\end{align}
and below we will bound the error by the worst case of the intrinsic dimension, namely $d=1$.  Using the heat kernel expansion \eqref{heatkernel}, we can expand the diffusion distance as,
\begin{align}\label{DD} D_t(x,y)^2 &=\left<e^{2t\Delta}\delta_x,\delta_x\right> + \left<e^{2t\Delta}\delta_y,\delta_y\right> -2 \left<e^{2t\Delta}\delta_x,\delta_y\right> = (8\pi t)^{-d/2} \left(2 - 2 e^{-d_g(x,y)^2/(8t)}\right)(1+\mathcal{O}(t,d_g(x,y))) \nonumber \\ 
&= (8\pi t)^{-d/2} \left(2 - 2 \left(1 - d_g(x,y)^2/(8t) + \mathcal{O}(d_g(x,y)^4/t^2) \right) \right)(1+\mathcal{O}(t,d_g(x,y))) \nonumber \\
&= (8\pi t)^{-d/2}(4t)^{-1}d_g(x,y)^2 (1 + \mathcal{O}(d_g(x,y)^2/t))(1+\mathcal{O}(t,d_g(x,y))) \nonumber \\
&=  \frac{d_g(x,y)^2}{(2\pi)^{d/2}(4t)^{d/2+1}} \left(1 + \mathcal{O}(t,d_g(x,y),d_g(x,y)^2/t) \right). \end{align}
Based on \eqref{DD} we define the rescaled diffusion map by,
\[ \hat\Phi_t(x) = (2\pi)^{d/4}(4t)^{d/4+1/2}\Phi_t(x), \]
and the rescaled diffusion distance, $\hat D_t(x,y) \equiv (2\pi)^{d/4}(4t)^{d/4+1/2}D_t(x,y)$, which is approximated by,
\[ \hat D_t(x,y)^2 \approx ||\hat\Phi_t(x) - \hat\Phi_t(y)||^2 \approx (2\pi)^{d/2}(4t)^{d/2+1}D_t(x,y)^2 = d_g(x,y)^2 + \mathcal{O}(t \, d_g(x,y)^2, d_g(x,y)^3, d_g(x,y)^4/t), \]
so that for $d_g(x,y)^2 \ll t \ll 1$ the rescaled diffusion distance approximates the geodesic distance.  

Notice that the rescaled diffusion distance only approximates the geodesic distance when the geodesic distance is small.  In particular, the diffusion distance should not be thought of as an approximate geodesic distance, as is clearly shown in Figure \ref{standardIDM} below.  For small distances, the rescaled diffusion distance and the geodesic distance also agree very closely with the Euclidean distance in any isometric embedding, as was shown in \cite{diffusion}.  In fact, the diffusion map provides a canonical embedding of the manifold $\mathcal{M}$ up to a rotation in the following sense: For any isometric image $\mathcal{N} = \iota(\mathcal{M})$ of $\mathcal{M}$ where $\iota$ is an isometry, the diffusion map embeddings of $\mathcal{N}$ and $\mathcal{M}$ with the same parameter $t$ will differ by at most an orthogonal linear map.  This is because the eigenfunctions of the Laplacian depend only on the geometry of the manifold, which is preserved by an isometric map, and for the eigenfunctions corresponding to repeated eigenvalues may differ only by an orthogonal transfomation.  Moreover, the fact that the rescaled diffusion map preserves small geodesic distances implies that the rescaled diffusion map is approximately an isometric embedding (this was shown previously in \cite{DMembed} which provides more detailed bounds).  In particular, this implies that, for $t$ small and $M$ large, the Laplace-Beltrami operator on $\Phi_t(\mathcal{M})$ is very close to the Laplace-Beltrami operator on $\mathcal{M}$.  One consequence of this fact is that if we iterate the rescaled diffusion map, the results should not change (up to a rotation).  

 \begin{figure}[h]
  \begin{center}
\includegraphics[width=.32\linewidth]{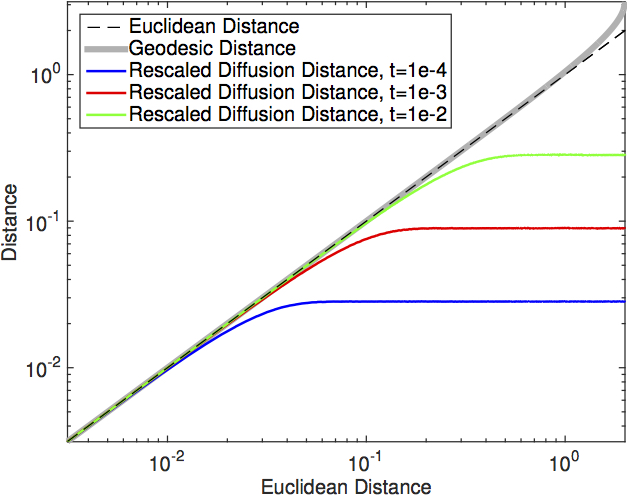}
 \includegraphics[width=.33\linewidth]{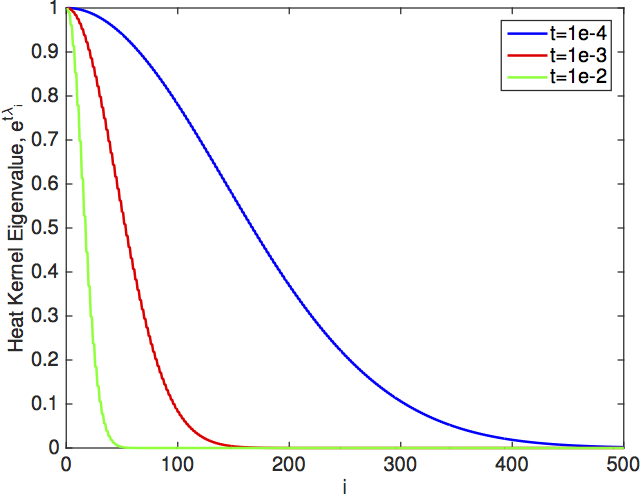}
 \includegraphics[width=.33\linewidth]{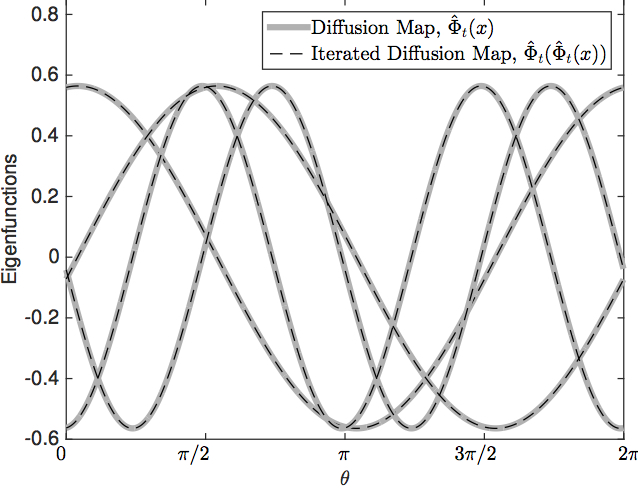}
\caption{\label{standardIDM} For 2000 data points equally spaced on a unit circle in $\mathbb{R}^2$, we compare the Euclidean distance, geodesic distance, and rescaled diffusion distances for $t \in \{10^{-4},10^{-3},10^{-2}\}$ (left).  We also show the spectra of the heat kernel $e^{t\lambda_i}$ for the corresponding values of $t$ (middle) and the results of iterating the standard diffusion map compared to the original diffusion map eigenfunctions (right).}
\end{center}
\end{figure}

To demonstrate these facts numerically, we generated $N=2000$ points $\{x_j\}_{j=1}^{N}$ equally spaced on a unit circle in $\mathbb{R}^2$.  We applied the diffusion maps algorithm to estimate the eigenvalues $\lambda_i$ and eigenfunctions $\varphi_i(x_j)$ of the Laplace-Beltrami operator on the unit circle.  We should emphasize that it is crucial to correctly normalize the eigenfunctions $\varphi_i$ using a kernel density estimate $q(x_j)$, that is, we require, 
\[ 1 = \frac{1}{N}\sum_{j=1}^N \frac{\varphi_i(x_j)^2}{q(x_j)} \approx \int_{\mathcal{M}} \varphi_i(x)^2 dV(x), \]
where $dV$ is the volume form on $\mathcal{M}$ inherited from the ambient space.  See \cite{BH14VB} for details on the Monte-Carlo integral above and a natural density estimate implicit to the diffusion maps construction.  We then evaluated the rescaled diffusion map $\hat \Phi_t(x_j)$ for $t \in \{10^{-4},10^{-3},10^{-2}\}$ and compared the resulting diffusion distances $\hat D_t(x_i,x_j) \approx ||\hat \Phi_t(x_i)-\hat\Phi_t(x_j)||$ to the Euclidean distances $||x_i-x_j||$ and the geodesic distances $d_g(x_i,x_j)$ in Figure \ref{standardIDM}.  Notice that for distances less than $t^{1/2}$ all the distances agree as shown above.  We also show the spectra of the heat kernels, $e^{t\lambda_i}$, which are the weights of the various eigenfunctions in the diffusion map embedding.  Notice that for $t$ large, the spectrum decays much faster, so fewer eigenfunctions are required for the diffusion distance to be well approximated by the Euclidean distance in the diffusion mapped coordinates.  Finally, for $t=10^{-2}$, we performed an `iterated' diffusion map, by computing the (rescaled) diffusion map of the data set $\hat x_j \equiv \hat\Phi_t(x_j)$, in effect finding $\hat\Phi_t(\hat\Phi_t(x))$.  We then compared the eigenfunctions $\varphi_i(x_j)$ from the first diffusion map with those $\hat\varphi_i(\hat x_j) = \hat\varphi_i(\hat\Phi_t(x))$.  Due to the symmetry of the unit circle, the eigenfunctions corresponding to repeated eigenvalues differed by an orthogonal linear map (meaning a phase shift in this case).  After removing the phase shift, the eigenfunctions are compared in Figure \ref{standardIDM}.  

Since the diffusion map $\Phi_t$ differs from the rescaled diffusion map $\hat\Phi_t$ by a scalar factor, the eigenfunction from iterating the standard diffusion map will also agree.  The only purpose of the rescaled diffusion map $\hat\Phi_t$ is to exactly recover the local distances in the data set, and thereby to also find the same eigenvalues (since rescaling the manifold will change the spectrum of the Laplacian).  Finally, if the standard diffusion map is used, the nuisance parameter $\epsilon$ will have to be retuned in order to iterate the diffusion map, since the diffusion distances will be scaled differently than the original distances.  We emphasize that the goal of this section is to show that iterating the standard diffusion map algorithm is not a useful method.  However, in \cite{LK} it was shown that a generalization of diffusion maps to \emph{local kernels} can be used to construct the Laplace-Beltrami operator with respect to a different metric.  In the remainder of the paper we will see that when the new metric is induced by a feature map on the data set, iterating the diffusion map has a nontrivial effect which can be beneficial. 

\subsection{Local Kernels and the Pullback Geometry}\label{section22}

The connection between kernel functions and geometry was introduced by Belkin and Niyogi in \cite{BN} and generalized by Coifman and Lafon in \cite{diffusion}.  Assuming that a data set $\{x_i\}$ is sampled from a density $p(x)$ supported on a $d$-dimensional manifold $\mathcal{M}\subset \mathbb{R}^m$, summing a function $\sum_i f(x_i)$ approximates a the integral $\int_{\mathcal{M}} f(y)p(y)\, dV(y)$ where $V(y)$ is the volume form on $\mathcal{M}$ inherited from the ambient space $\mathbb{R}^m$.  The central insight of \cite{BN,diffusion} is that by choosing a kernel function $K(x_i,x_j)$ which has exponential decay, the integral $\sum_{j}K(x_i,x_j)f(x_j) \approx \int_{\mathcal{M}} K(x_i,y)f(y)p(y)\, dV(y)$ is localized to the tangent space $T_{x_i}\mathcal{M}$ of the manifold.  

The theory of \cite{BN,diffusion} was recently generalized in \cite{LK} to a wide class of kernels called \emph{local kernels} which are assumed only to have decay that can be bounded above by an exponentially decaying function of distance.  The results of \cite{LK} generalized an early result of \cite{SingerAnisotropy2008} to a much wider class of kernels and connected these early results to their natural geometric interpretations.  In this paper we will use the following prototypical example of a local kernel, since it was shown in \cite{LK} that every operator which can be obtained with a local kernel can also be obtained with a prototypical kernel.  Let $C(x)$ be a matrix valued function on the manifold $\mathcal{M}$ such that each $C(x)$ is a symmetric positive definite $m\times m$ matrix.  Define the prototypical kernel with covariance $C$ (and first moment of zero) by
\begin{align}\label{PK} K(\epsilon,x,y) = \exp\left(-\frac{(x-y)^T C(x)^{-1}(x-y)}{2\epsilon}\right). \end{align}
The theory of local kernels \cite{LK} uses a method closely related to the method of Diffusion Maps of \cite{diffusion} to construct matrices $L_{\epsilon}$ and $L^*_{\epsilon}$ which are discrete approximations to the following operators,
\begin{equation} \label{eqScriptL}
 \mathcal{L}f = \frac{1}{2}c_{ij}\nabla_i \nabla_j f  \hspace{40pt} \mathcal{L}^*f = \frac{1}{2} \nabla_j \nabla_i(c_{ij} f),
 \end{equation}
 in the sense that in the limit of large data and as $\epsilon\to 0$ we have $L_{\epsilon}\to \mathcal{L}$ and $L^*_{\epsilon} \to \mathcal{L}^*$.
Notice that the matrix valued function $C(x)$ acts on the ambient space, whereas the tensor $c(x)$ in the limiting operator $\mathcal{L}$ is only defined on the tangent planes of $\mathcal{M}$.  As shown in \cite{LK}, only the projection of $C(x)$ onto the tangent space $T_x\mathcal{M}$ will influence the operator $\mathcal{L}$.  Thus, we introduce the linear map $\mathcal{I}(x):\mathbb{R}^m \to T_x\mathcal{M}$ which acts as the identity on the tangent plane as a subspace of $\mathbb{R}^m$ and sends all vectors originating at $x$ which are orthogonal to $T_x\mathcal{M}$ to zero.  The map $\mathcal{I}$ projects the ambient space onto the tangent space so that $\mathcal{I}(x)$ is a $d\times m$ matrix and we define $c(x) =\mathcal{I}(x)C(x)\mathcal{I}(x)^\top$.

Given data sampled from a $d$-dimensional manifold $\mathcal{M}$ embedded in Euclidean space $\mathbb{R}^m$ the manifold $\mathcal{M}$ naturally inherits a Riemannian metric, $g_{\cal M}$, from the ambient space.  The standard Diffusion Maps algorithm uses an isotropic kernel (where the covariance matrix is a multiple of the identity matrix) to estimate the Laplace-Beltrami operator corresponding to the metric $g_{\cal M}$.  It was shown in \cite{LK} that local kernels such as \eqref{PK} can be used to approximate the Laplace-Beltrami operator corresponding to a new Riemannian metric $\tilde g = c^{-1/2}g_{\cal N}c^{-1/2}$, where $g_{\cal N}$ is a Riemannian metric of $\cal{N}=\mathcal{H}(\mathcal{M})$ for diffeomorphism $\mathcal{H}$ that satisfies $c^{-1}=D\mathcal{H}^\top D\mathcal{H}$. Formally, we summarize this result as follows:
 
\begin{thm}[Pullback geometry of local kernels, with nonuniform sampling]\label{maintheorem} Let $(\mathcal{M}, g_{\mathcal{M}})$ be a Riemannian manifold and let $\{x_i\}_{i=1}^N \subset \mathcal{M}$ be sampled according to any smooth density on $\mathcal{M}$.  Let $\mathcal{H}:\mathcal{M} \to \mathcal{N}$ be a diffeomorphism and let $y_i = \mathcal{H}(x_i)$ and $c(x_i)^{-1} = D\mathcal{H}(x_i)^\top D\mathcal{H}(x_i)$.  For the local kernel $K$ in \eqref{PK}, define the symmetric kernel $\overline K(\epsilon,x,y) = K(\epsilon,x,y) + K(\epsilon,y,x)$.  Then for any smooth function $f$ on $\mathcal{M}$,
\[ \lim_{N\to\infty} \frac{2}{\epsilon}\left(  \frac{{\displaystyle\sum_{j=1}^N} { \overline K(\epsilon,x_i,x_j)f(x_j)}/{\sum_l \overline K(\epsilon,x_j,x_l)}}{{\displaystyle \sum_{j=1}^N} { \overline K(\epsilon,x_i,x_j)}/{\sum_l \overline K(\epsilon,x_j,x_l)}}-f(x_i)\right) = \Delta_{\tilde g}f(x_i) + \mathcal{O}(\epsilon)= \Delta_{g_{\mathcal{N}}} (f\circ \mathcal{H}^{-1})(y_i) + \mathcal{O}(\epsilon)\]
where $\tilde g(u,v) = g_{\mathcal{N}}(D\mathcal{H}u,D\mathcal{H}v)$.
\end{thm}
Theorem \ref{maintheorem} follows directly from Theorem 4.7 of \cite{LK}.  This result was used by \cite{LK} to represent a diffeomorphism between two manifolds.  We assume we are given a training data set $x_i \in \mathcal{M}\subset \mathbb{R}^m$ sampled from the data manifold $\mathcal{M}$ along with the true feature values, $y_i = \mathcal{H}(x_i)$, where $y_i$ lie on $\mathcal{N} = \mathcal{H}(\mathcal{M})$.  When $\mathcal{H}$ is a diffeomorphism, we can use a local kernel to pullback the Riemannian metric from $\mathcal{N}$ onto $\mathcal{M}$ via the correspondence between the data sets.  With this metric on $\mathcal{M}$, the two manifolds are isometric, which implies that the Laplacians ($\Delta_{\tilde g}$ on $\mathcal{M}$ and $\Delta_{g_{\mathcal{N}}}$ on $\mathcal{N}$) have the same eigenvalues, and that the associated eigenfunctions of any eigenvalue are related by an orthogonal transformation \cite{laplacianBook}.  

In Section \ref{section3} we will give a rigorous method to approximate $c(x_i)^{-1} = D\mathcal{H}(x_i)^\top D\mathcal{H}(x_i)$ using the training data. With this approximation, numerically we evaluate the local kernel
\begin{align}\label{diffeokernel} K(\epsilon,x_i,x_j) = \exp\left(-\frac{|| D\mathcal{H}(x_i)(x_j-x_i)||^2}{2\epsilon}\right). \end{align}
By Theorem \ref{maintheorem}, using the kernel \eqref{diffeokernel}, we approximate the Laplacian $\Delta_{\tilde g} = {(\mathcal{H}^{-1})}^*\Delta_{g_{\mathcal{N}}}$ on $\mathcal{M}$. Simultaneously, using the standard diffusion maps algorithm (with $\alpha=1$) we approximate the Laplacian $\Delta_{g_{\mathcal{N}}}$ on $\mathcal{N}$.  Since $(\mathcal{M},\tilde g)$ and $(\mathcal{N},g_{\mathcal{N}})$ are isometric, the eigenvalues of $\Delta_{\tilde g}$ and $\Delta_{g_{\mathcal{N}}}$ will be the same and the corresponding eigenfunctions will be related by an orthogonal transformation.  By taking sufficiently many eigenfunctions $\varphi_l$ and $\tilde \varphi_l$ on the respective manifolds, the eigenfunctions can be considered coordinates of an embeddings $\Phi(x) = (\varphi_1(x),...,\varphi_{M}(x))^\top$ and $\tilde \Phi(y) = (\tilde \varphi_1(y),...,\tilde \varphi_{M}(y))^\top$.  We can now project the diffeomorphism $\mathcal{H}$ into these coordinates as,
  \begin{center}
$$  \begin{array}[c]{ccc}
\mathcal{M}&\xrightarrow{\ \ \ \ \ \mathcal{H}\ \ \ \ \ }&\mathcal{N}\\ \\
\left\downarrow\rule{0cm}{.5cm}\right.   \scriptstyle{\Phi}&&\left\downarrow\rule{0cm}{.5cm}\right. \scriptstyle{\tilde\Phi}\\ \\
L^2(\mathcal{M},\tilde g) \approx \mathbb{R}^{M}&\xrightarrow{\ \ \ \ \ H\ \ \ \ \ }&L^2
(\mathcal{N},g_{\mathcal{N}}) \approx \mathbb{R}^{M}
\end{array}$$
\end{center}
where $H = \tilde\Phi \circ \mathcal{H} \circ  \Phi^{-1}$ is linear and can be estimated using linear least squares.  Finally, to extend the diffeomorphism to new data points $x \in \mathcal{M}$ we need only extend the map $\Phi$ to this new data point using standard methods such as the Nystr\"om extension.  

Notice that the key to the existence of the linear map $H$ is that the diffeomorphism $\mathcal{H}$ induces a new metric on $\mathcal{M}$ that is isometric to the metric on $\mathcal{N}$.  In Section \ref{section4} we will make use of this theorem for identifying feature in $\mathcal{M}$ that is relevant to the data in $\mathcal{N}$, even when $\mathcal{H}$ is not a diffeomorphism, but simply a mapping.  However, we will first give rigorous results in Section \ref{section3} for approximating the tangent plane $T_x\mathcal{M}$ and the derivative $D\mathcal{H}$ from data.

\section{Tangent Spaces and Derivatives}\label{section3}

Section \ref{section22} shows that to build a global map $\mathcal{H}:\mathcal{M} \to\mathcal{N} =\mathcal{H}(\mathcal{M})$ between data sets, we need to estimate the local linear maps $D\mathcal{H}(x_i)$ between the tangent spaces $T_{x_i}\mathcal{M}$ and $T_{\mathcal{H}(x_i)}\mathcal{N}$ at each point $x_i \in \mathcal{M}$. Notice that $D\mathcal{H}(x)$ is a $d_{\mathcal{N}} \times d$ matrix, where $d$ is the intrinsic dimension of $\mathcal{M}$ and $d_{\mathcal{N}}$ is the intrinsic dimension of $\mathcal{N}$.  However, it will be more natural to represent $D\mathcal{H}$ as a map between the ambient spaces $\mathbb{R}^m\supset T_x\mathcal{M}$ and $\mathbb{R}^n \supset T_{\mathcal{H}(x)}\mathcal{N}$.  Recall that $\mathcal{I}(x)$, is a $d \times m$ matrix valued function which projects from the ambient space $\mathbb{R}^m$ onto the tangent space $T_x\mathcal{M} \subset \mathbb{R}^m$ such that $\mathcal{I}(x) \mathcal{I}(x)^\top = I_{d\times d}$.  We introduce the notation $\mathcal{I}_{\mathcal{N}}(\mathcal{H}(x))$ for the $d_{\mathcal{N}} \times n$ matrix valued function given by the projection from $\mathbb{R}^n$ onto the tangent space $T_{\mathcal{H}(x)}\mathcal{N} \subset \mathbb{R}^n$ such that $\mathcal{I}_{\mathcal{N}}(\mathcal{H}(x))\mathcal{I}_{\mathcal{N}}(\mathcal{H}(x))^\top = I_{d_{\mathcal{N}}\times d_{\mathcal{N}}}$.  With this notation, 
\begin{align}\label{DhatH} D{\hat{\mathcal{H}}}(x) = \mathcal{I}_{\mathcal{N}}(\mathcal{H}(x))^\top D\mathcal{H}(x) \mathcal{I}(x) . \end{align}
In practice we will estimate $D\hat{\mathcal{H}}(x) \in \mathbb{R}^{n\times m}$, however, when used to construct a local kernel as in Section \ref{section22} only $D\mathcal{H}$ will influence the intrinsic geometry defined by the kernel.  

In this section we improve and make rigorous a method originally introduced in \cite{LK} that estimates the local linear maps from data using a weighted regression.  To estimate $D\hat{\mathcal{H}}(x_i)$, we take the nearest neighbors $\{x_j\}$ of $x_i$ and use the correspondence to find $y_i = \mathcal{H}(x_i)$ and the neighbors $y_j = \mathcal{H}(x_j)$.  Note that $y_j$ may not be the nearest neighbors of $y_i$ due to the distortion of the geometry introduced by $\mathcal{H}$; although if $\mathcal{H}$ is a diffeomorphism (as in \cite{LK}) the local distortion will be very small.  In \cite{LK} they construct the weighted vectors
\[ dx_j = \exp\left(-||x_j-x_i||^2/(4\epsilon)\right)(x_j-x_i) \hspace{50pt} dy_j = \exp\left(-||x_j-x_i||^2/(4\epsilon)\right)(y_j-y_i), \] 
and define $D\hat{\mathcal{H}}(x_i)$ to be the matrix which minimizes $\sum_j || dy_j - D\hat{\mathcal{H}}(x_i)dx_j ||^2$.  Intuitively, the exponential weight is used to localize the vectors; otherwise the linear least squares problem would try to preserve the longest vectors $x_j-x_i$, which do not represent the tangent space well.  This method of localization was used in \cite{LK} for estimating $D\hat{\mathcal{H}}(x_i)$, and it is also closely related to a method of determining the tangent space of a manifold which was introduced in \cite{singerWu}.  Using the foundational theory developed in \cite{diffusion} we will now make this method of finding tangent spaces and derivatives rigorous.

\begin{thm}\label{thm1}
Let $x_i$ be samples from $\mathcal{M} \subset \mathbb{R}^m$ with density $p(x)$ and $y_i=\mathcal{H}(x_i)$ where $\mathcal{H}:{\cal M}\rightarrow\mathcal{N} \subset\mathbb{R}^n$. Define $X$ to be a matrix with columns $X_j = D(x)^{-1/2}\exp\left(-\frac{||x_j - x ||^2}{4\epsilon}\right)(x_j-x) = D(x)^{-1/2}dx_j$ and let $Y$ be a matrix with columns $Y_j = D(x)^{-1/2}\exp\left(-\frac{||x_j - x ||^2}{4\epsilon}\right)(y_j-y) = D(x)^{-1/2}dy_j$, where 
\begin{align}
D(x) = \sum_{i=1}^N \exp\left(-\frac{||x_i - x ||^2}{2\epsilon}\right).\nonumber
\end{align}
Then, 
\begin{align}
\lim_{N\rightarrow\infty}  \frac{1}{\epsilon} Y X^{\top} = D\hat{\mathcal{H}}(x) + \epsilon  R_{\mathcal{H}}(x) + \mathcal{O}(\epsilon^2),
\end{align}
with $D\hat{\mathcal{H}}(x)$ as in \eqref{DhatH} and $R_{\mathcal{H}}(x)\in\mathbb{R}^{n\times m}$. 
\end{thm}

\begin{proof}
Following Appendix B of \cite{diffusion}, let $x,y \in \mathcal{M}$ with $||y-x|| < \sqrt{\epsilon}$ with $\epsilon$ sufficiently small so that there is a unique geodesic $\gamma:[0,s] \to \mathcal{M}$ with $\gamma(0)=x$ and $\gamma(s)=y$.  Let $\{e_i\}$ be a basis for the tangent space $T_x\mathcal{M}$ and define the projection of the geodesic onto the tangent plane by $u_i = \left<y-x,e_i \right> = \left<\gamma(s)-\gamma(0),e_i \right>$.  Locally, we can parameterize the manifold using a function $q:T_x\mathcal{M} \to T_x\mathcal{M}^{\perp}$ so that $y-x = (u,q(u))$.  We now use the Taylor expansion $\gamma(s) = \gamma(0) + s\gamma'(0) + s^2 \gamma''(0)/2 + \mathcal{O}(s^3)$, where $\gamma'(0) \in T_x\mathcal{M}$ and $\gamma''(0)$ is orthogonal to the tangent space.  Combining the previous lines yields, 
\[ (u,q(u)) = y-x = \gamma(s)-\gamma(0) = s\gamma'(0) + s^2 \gamma''(0)/2 + \mathcal{O}(\epsilon^{3/2})  \]
which implies that $u = s\gamma'(0) + \mathcal{O}(\epsilon^{3/2})$ and $q(u) = s^2 \gamma''(0)/2 + \mathcal{O}(\epsilon^{3/2})$.  From Equation (B.2) in \cite{diffusion}, we have $||y-x||^2 = ||u||^2 + \mathcal{O}(\epsilon^{2})$.  For $v \in T_{x}\mathcal{M}$ and $w \in T_x\mathcal{M}^{\perp}$ we have,
\[ \left<y-x,v\right> = s \left<\gamma'(0),v \right> + \mathcal{O}(\epsilon^{3/2}) \hspace{30pt} \left<y-x,w\right> = s^2/2 \left<\gamma''(0),w \right> + \mathcal{O}(\epsilon^{3/2}) \]
This shows that taking the inner product with vectors $y-x$ in the $\sqrt{\epsilon}$ neighborhood of $x$, vectors in the tangent space are of order-$\sqrt{\epsilon}$ and vectors in the orthogonal complement are of order-$\epsilon$.  

Let $\{x_i\}$ be discrete data points sampled from $\mathcal{M}$.  Recall from \cite{diffusion} we have,
\begin{align}\label{sumscaling} \lim_{N\to\infty}\frac{1}{N}D(x) &\equiv \lim_{N\rightarrow\infty}\frac{1}{N}\sum_{i=1}^N \exp\left(-\frac{||x_i - x ||^2}{2\epsilon}\right) =  \int_{\mathcal{M}} \exp\left(-\frac{||y - x ||^2}{2\epsilon}\right) p(y) \, dV(y)\nonumber \\ &= \int_{T_x\mathcal{M}} \exp\left(-\frac{||u||^2}{2\epsilon}\right)p(x) (1+ \mathcal{O}(\epsilon)) \, du 
= (2\pi\epsilon)^{d/2}p(x) + \mathcal{O}(\epsilon^{d/2 + 1}), 
\end{align}
where the continuous integral is a result of taking Monte-Carlo limit over data sampled from the sampling density $p(y)$ with respect to the volume form $dV$ that $\mathcal{M}$ inherits from the ambient space. The restriction of the integral to the tangent plane $T_x\mathcal{M}$ was shown in \cite{diffusion} and follows from the exponential decay of the integrand and we also use the fact from \cite{diffusion} that $dV(y) = (1+\mathcal{O}(\epsilon))du$.  Finally, the change of variables in \eqref{sumscaling} drops all the odd order terms due to the symmetry of the kernel.

Recall that $X$ was the matrix with columns $X_j = D(x)^{-1/2}\exp\left(-\frac{||x_j - x ||^2}{4\epsilon}\right)(x_j-x) = D(x)^{-1/2}dx_j$ and $Y$ is the matrix with columns 
$Y_j = D(x)^{-1/2}\exp\left(-\frac{||x_j - x ||^2}{4\epsilon}\right)(y_j-y) = D(x)^{-1/2}dy_j$. For any vectors $v \in \mathbb{R}^m$ and $w \in \mathbb{R}^n$ we have,
\begin{align}\label{DH1} 
\lim_{N\rightarrow\infty} w^\top Y X^{\top}v &= \lim_{N\rightarrow\infty} D(x)^{-1}\sum_{j=1}^N \exp\left(-\frac{||x_j - x ||^2}{2\epsilon}\right)  \left<\mathcal{H}(x_j)-\mathcal{H}(x),w\right>\left<x_j-x,v \right> \nonumber \\
&= \lim_{N\rightarrow\infty} \left(\frac{D(x)}{N}\right)^{-1}\frac{1}{N}\sum_{j=1}^N \exp\left(-\frac{||x_j - x ||^2}{2\epsilon}\right)  \left<\mathcal{H}(x_j)-\mathcal{H}(x),w\right>\left<x_j-x,v \right> \nonumber \\
&= (2\pi\epsilon)^{-d/2}p(x)^{-1}(1+\mathcal{O}(\epsilon))\int_{\mathcal{M}} \exp\left(-\frac{||y - x ||^2}{2\epsilon}\right)  \left<\mathcal{H}(y)-\mathcal{H}(x),w\right>\left<y-x,v \right> p(y) \, dV(y) \nonumber \\
&= (2\pi\epsilon)^{-d/2}\int_{T_x\mathcal{M}} \exp\left(-\frac{||u||^2}{2\epsilon}\right)  \left<D\mathcal{H}(x) u + \frac{1}{2}u^\top H({\cal H})(x)u+\mathcal{O}(\epsilon^2),w\right> \left(\left<u,v\right> + \left<q(u),v\right>\right) (1+\mathcal{O}(\epsilon)) \, du
\end{align}
where $H(\cdot)$ is the Hessian operator and the last equality follows from using the exponential decay of the integrand to restrict the integral to the tangent plane (see \cite{diffusion} for details).
For $w \in T_{\mathcal{H}(x)}\mathcal{H}(\mathcal{M})$ and $v \in T_x\mathcal{M}$ we reduce \eqref{DH1} to,
\begin{align}\label{dfpar} 
\lim_{N\rightarrow\infty} w^\top Y X^{\top}v &=  (2\pi\epsilon)^{-d/2}\int_{T_x\mathcal{M}} \exp\left(-\frac{||u||^2}{2\epsilon}\right) \sum_{i,j,k} D\mathcal{H}(x)_{ij}u_j w_i u_k v_k \, du + \mathcal{O}(\epsilon^{2}) = \epsilon \sum_{i,j} D\mathcal{H}(x)_{ij}w_i v_j + \mathcal{O}(\epsilon^{2}) \nonumber \\ 
&= \epsilon w^\top D\mathcal{H}(x) v + \mathcal{O}(\epsilon^{2}) \end{align}
On the other hand, for $w \in \mathbb{R}^n$ and $v \in T_x\mathcal{M}^\perp$ we reduce \eqref{DH1} to,
\begin{align} \lim_{N\rightarrow\infty} w^\top Y X^{\top}v &=  (2\pi\epsilon)^{-d/2}\int_{T_x\mathcal{M}}\frac{1}{2}\exp\left(-\frac{||u||^2}{2\epsilon}\right) \sum_{i,j,k,l}  [H({\cal H}_l)(x)]_{ij}u_iu_j w_l q_k(u) v_k (1+ \mathcal{O}(\epsilon)) \, du  \nonumber \\
&=  (2\pi\epsilon)^{-d/2}\int_{T_x\mathcal{M}} \frac{1}{4}\exp\left(-\frac{||u||^2}{2\epsilon}\right)  \sum_{i,j,k,l,a,b}  [H({\cal H}_l)(x)]_{ij}u_iu_ju_au_b w_l [H(q_k)(0)]_{ab} v_k\, du + \mathcal{O}(\epsilon^{3}) \nonumber \\   
&= \epsilon^2 \sum_{k,l} v_k w_l R_{\mathcal{H}}(x)_{lk}+ \mathcal{O}(\epsilon^{3}) = \epsilon^2 w^\top R_{\mathcal{H}}(x) v + \mathcal{O}(\epsilon^{3}),\label{dfortho} \end{align}
where we have used the expansion $q_k(u) = u^\top H(q_k)(0) u$ and we define
\begin{align}
R_{\mathcal{H}}(x)_{lk} = \frac{1}{4}\Big(\sum_{i,j}[H({\cal H}_l)(x)]_{ii}[H(q_k)(0)]_{jj} + [H({\cal H}_l)(x)]_{ij}[H(q_k)(0)]_{ij} + [H({\cal H}_l)(x)]_{ij}[H(q_k)(0)]_{ji}\Big).\label{RH}
\end{align}
Finally, it is easy to see that for  $w \in T_{\mathcal{H}(x)}\mathcal{H}(\mathcal{M})^\perp$ and $v \in T_x\mathcal{M}$ all the terms will be polynomials of degree 3 in the coordinates of $u$, and since these terms are all odd, by the symmetry of the domain of integration we have $\lim_{N\rightarrow\infty}  w^\top YX^\top v = \mathcal{O}(\epsilon^{3})$.  Together with \eqref{dfpar} and \eqref{dfortho}, the proof is complete.
\end{proof}  

We note that the above proof can easily be generalized on kernels of the form $K(\epsilon,x,y) = h\left(\frac{||y-x||^2}{\epsilon}\right)$ for $h:[0,\infty) \to [0,\infty)$ having exponential decay by following \cite{diffusion}. In the remainder of this section, we will discuss the consequences of this result in more details. In particular, we shall see that the scaling law established in this theorem provides systematic methods to identify tangent spaces, estimate derivavtive $D\mathcal{H}$, as well as to estimate the kernel bandwidth parameter $\epsilon$, which is crucial for accurate numerical approximation.

\subsection{Identifying Tangent Spaces with the Singular Value Decomposition}\label{svd}

The first method of leveraging Theorem \ref{thm1} is with the singular value decomposition (SVD).  Intuitively, the singular vectors will naturally be sorted into tangent vectors, with singular values of order $\sqrt{\epsilon}$, and orthogonal vectors, with singular values of order $\epsilon$.  To see this we state the following corollary to Theorem \ref{thm1}.

\begin{cor}\label{cor1}
Let $x_i$ be samples from ${\cal M}\subset{R}^m$ with density $p(x)$. Define $X$ to be a matrix with columns $X_j = D(x)^{-1/2}\exp\left(-\frac{||x_j - x ||^2}{4\epsilon}\right)(x_j-x)= D(x)^{-1/2}dx_j$, where $D(x)$ is defined as in Theorem~\ref{thm1}. Then, 
\begin{align}
\lim_{N\rightarrow\infty}  \frac{1}{\epsilon} XX^{\top} = \mathcal{I}(x)^{\top}\mathcal{I}(x) + \epsilon  R_{\mathcal{I}}(x) + \mathcal{O}(\epsilon^2),\label{limitcor}
\end{align}
where $R_{I}(x)\in\mathbb{R}^{m\times m}$. 
\end{cor}

\begin{proof}
The proof follows from Theorem~\ref{thm1} with ${\cal H}(x)=x$ so that $D\mathcal{H}(x) = I_{d \times d}$ and $D\hat{\mathcal{H}}(x) = \mathcal{I}(x)^{\top}D\mathcal{H}(x)\mathcal{I}(x) = \mathcal{I}(x)^{\top}\mathcal{I}(x)$.  Note that the Hessian $H(\mathcal{H})$ in the definition of $R_{\mathcal{H}}$ in \eqref{RH} is with respect to the coordinates $u \in T_x\mathcal{M}$, so in general $R_{\mathcal{I}}$ is not necessarily zero.
In fact, by repeating the argument in the derivation of \eqref{RH} one can show that,
\begin{align}
R_{\mathcal{I}}(x)_{lk} = \frac{1}{4}(\mathcal{I}^\perp(x))^\top\Big(\sum_{i,j}[H(q_l)(0)]_{ii}[H(q_k)(0)]_{jj} + [H(q_l)(0)]_{ij}[H(q_k)(0)]_{ij} + [H(q_l)(0)]_{ij}[H(q_k)(0)]_{ji}\Big)\mathcal{I}^\perp(x),\nonumber
\end{align}
where $\mathcal{I}^\perp(x):\mathbb{R}^m\to T_x\mathcal{M}^\perp$ is a projection operator that is identity in the directions orthogonal to $T_x\mathcal{M}$ and maps all vectors originating at $x$ to zero when they are in $T_x\mathcal{M}$.
\end{proof}

Recall that $\mathcal{I}(x):\mathbb{R}^m \to T_x\mathcal{M}$ is the projection onto the tangent space at $x$ viewed as a subspace of $\mathbb{R}^m$.  Corollary \ref{cor1} suggests that if $v\in T_x\mathcal{M}$, then $\lim_{N\to\infty}v^\top XX^\top v = \epsilon ||v||^2 + \mathcal{O}(\epsilon^2)$, whereas for $v \in T_x\mathcal{M}^{\perp}$ we find $\lim_{N\to\infty}v^\top XX^\top v = \epsilon^2 v^\top R_I(x) v + \mathcal{O}(\epsilon^3) = \mathcal{O}(\epsilon^2 ||v||^2)$.
This shows that if $v$ is a singular vector, the associated singular value, 
\[ \sigma_v = \lim_{N\to\infty}\frac{\sqrt{v^\top X X^\top v}}{||v||}, \] 
will either be order-$\sqrt{\epsilon}$ if $v$ is in the tangent space, or order-$\epsilon$ if $v$ is orthogonal to the tangent space.  Since the singular value decomposition of $X$ finds $v$ which maximizes $\sigma_v$, when $\epsilon$ is well tuned the first $d$ singular values will all be order-$\sqrt{\epsilon}$ and the remaining $m-d$ singular values will be order-$\epsilon$.  This fact gives us a way to identify the tangent vectors of the manifold by defining the scaling law, $\alpha_l$, of a singular value, $\sigma_l$, to be the exponential power such that $\sigma_l \propto \epsilon^{\alpha_l}$.  When $\alpha_l \approx 1/2$ then the associated singular vector is a tangent vector and when $\alpha_l \geq 1$ then the associated singular vector is orthogonal to $T_x\mathcal{M}$.  

 \begin{figure}[h]
  \begin{center}
\includegraphics[width=.5\linewidth]{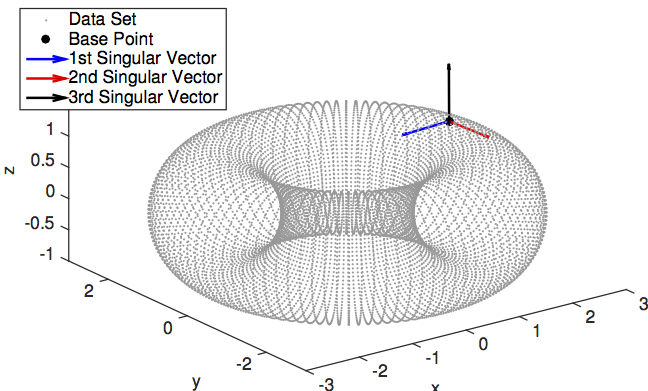}
 \includegraphics[width=.4\linewidth]{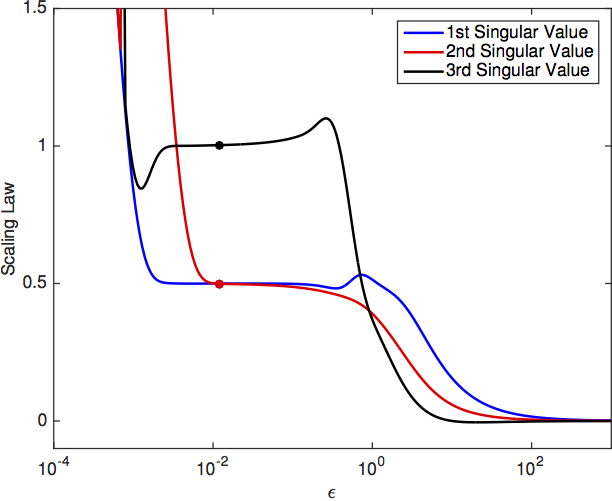} \\
 \includegraphics[width=.5\linewidth]{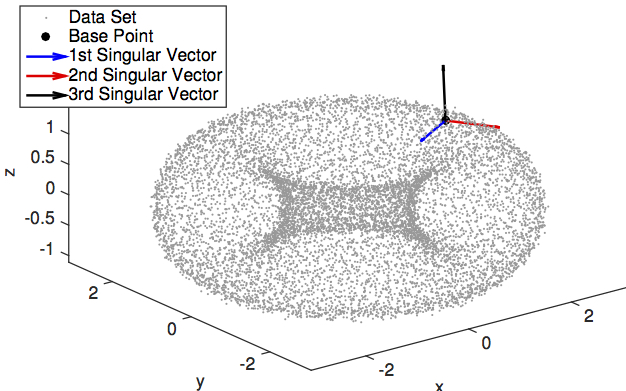}
 \includegraphics[width=.4\linewidth]{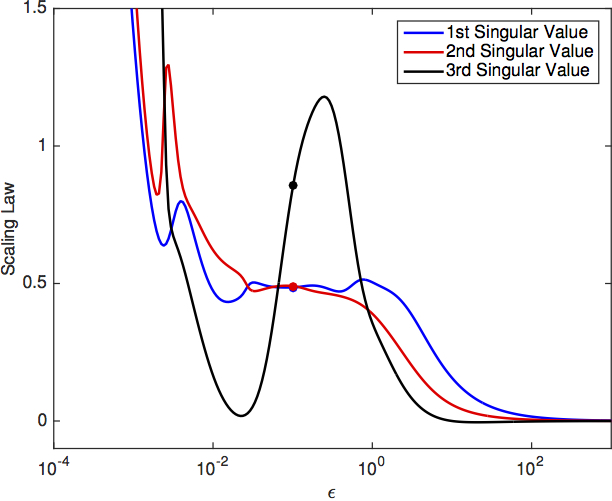}
\caption{\label{svdscaling} Data set sampled from a Torus embedded in $\mathbb{R}^3$ (top) and with noise added (bottom).  Singular vectors are shown (left) that correspond to the optimal choice of $\epsilon$ (shown above with the solid dot in the scaling law curves, see Section \ref{tuning}) based on the empirical scaling laws (right) for the various singular values and the determinant of the weighted vectors $X$ at the base point $(1.996,0.126,1.000)^\top$.}
\end{center}
\end{figure}

For discrete data, this power law will change as a function of the bandwidth parameter $\epsilon$. Numerically, we can estimate this power law by computing $\sigma_l(\epsilon)$ for discrete values $\epsilon_i$ and then approximating,
\[ \alpha_l = \frac{d\log(\sigma_l)}{d\log(\epsilon)} \approx \frac{\log(\sigma_l(\epsilon_i))-\log(\sigma_l(\epsilon_{i-1}))}{\log(\epsilon_i)-\log(\epsilon_{i-1})} \]
We now demonstrate this numerically by sampling 10000 points $(\theta_i,\phi_i) \in [0,2\pi]^2$ from a uniform grid and mapping them onto a torus embedded in $\mathbb{R}^3$ by $(x,y,z)^\top = ((2+\cos(\theta))\cos(\phi),(2+\cos(\theta))\sin(\phi),\sin(\theta))^\top$.  We chose a point $x = (1.996,0.126,1.000)^\top$ and constructed the weighed vectors $X_j = D(x)^{-1/2}\exp\left(-\frac{||x_j - x ||^2}{4\epsilon}\right)(x_j-x)$ for $\epsilon_l = 2^{-13+l/10}$ where $l=1,...,230$.  For each value of $\epsilon_l$ we compute the three singular values of $X_j$ and then we compute the scaling laws for each singular value. These scaling laws are shown in Figure \ref{svdscaling}.  We selected the optimal value of $\epsilon$ using the method that we will describe in Section \ref{tuning}, which are highlighted by a solid dot in the scaling law curves, and we plot the associated singular vectors in Figure \ref{svdscaling}.  

To demonstrate the robustness of this methodology to small noise in the ambient space, we repeated the experiment adding a three dimensional Gaussian random perturbation with mean zero and variance $0.04 I_{3\times 3}$ to each point.  In the noisy case, the theoretical scaling laws are obtained for a much smaller range of values of $\epsilon$ as shown in Figure \ref{svdscaling}.  In fact, when analyzed at a small scale ($\epsilon < 0.15$) all three singular values have scaling law $\alpha_l \approx 1/2$, which represents the three dimensional nature of the manifold after the addition of the noise.  However, the scaling laws also capture the approximate two-dimensional structure, as shown by the scaling law of the third singular vector being very close to $1$ for $0.22 < \epsilon < 0.4$.  This suggests that the scaling laws are robust for perturbations of magnitude less than $\epsilon$, however, the singular vectors are more sensitive as shown by the slight tilt in the tangent plane defined by the first two singular vectors in Figure \ref{svdscaling}.

\subsection{Estimating Derivatives with the Linear Regressions}\label{linreg}

We now return to the problem of estimating the derivative of a nonlinear mapping $\mathcal{H}:\mathcal{M} \subset \mathbb{R}^m \to\mathcal{N} \subset \mathbb{R}^n$ where we assume that we know the values of $\mathcal{H}$ on our training data set $y_i = \mathcal{H}(x_i)$.  As mentioned above, the approach of \cite{LK} was to use a linear regression to estimate $D\hat{\mathcal{H}}(x_i)$ as the matrix which minimizes $\sum_j || dy_j - D\hat{\mathcal{H}}(x_i) dx_j ||^2$.  Using the theory developed in Section \ref{svd} we can now rigorously justify this approach.  Notice that the linear regression minimizes the error $Y \approx D\hat{\mathcal{H}}(x_i)  X$ by setting $D\hat{\mathcal{H}}(x_i) \equiv YX^\top(XX^\top)^{-1}$ (where the additional factor of $D(x)$ from Theorem \ref{thm1} cancels making this equivalent to the approach of \cite{LK}).  

Theorem \ref{thm1} suggested that a simple method of estimating the derivative $ D\hat{\mathcal{H}}(x)$ is with the correlation matrix $\frac{1}{\epsilon}YX^\top$.  Numerically, we found that a better estimate of $D\hat{\mathcal{H}}(x)$ is given by the linear regression $YX^{\top}(XX^{\top})^{-1}$, and we also analyze this construction.  From Corollary \ref{cor1} we have $\lim_{N\to\infty}XX^\top = \epsilon\mathcal{I}(x)^\top\mathcal{I}(x) + \epsilon^2 R_I(x) + \mathcal{O}(\epsilon^3)$, which implies that in the limit of large data,
\[ \lim_{N\to\infty}(XX^\top)^{-1} = \frac{1}{\epsilon}((\mathcal{I}(x)^\top \mathcal{I}(x))^{\dagger} - \epsilon R_I(x) + \mathcal{O}(\epsilon^2)), \]
where $\dagger$ denotes the pseudo-inverse.
 Combining the results of Theorem \ref{thm1} and Corollary \ref{cor1} we have,
\begin{align}\lim_{N\to\infty} YX^\top (XX^\top)^{-1} &= (D\hat{\mathcal{H}}(x) + \epsilon R_{\mathcal{H}}(x) + \mathcal{O}(\epsilon^2))((\mathcal{I}(x)^\top \mathcal{I}(x))^{\dagger} - \epsilon R_I(x) + \mathcal{O}(\epsilon^2)) \nonumber \\
&= D\hat{\mathcal{H}}(x)(\mathcal{I}(x)^\top \mathcal{I}(x))^{\dagger}+ \mathcal{O}(\epsilon) \nonumber \\
&= \mathcal{I}_{\mathcal{N}}(\mathcal{H}(x))^\top D\mathcal{H}(x) \mathcal{I}(x)(\mathcal{I}(x)^\top \mathcal{I}(x))^{\dagger}+ \mathcal{O}(\epsilon) \nonumber \\ 
&= \mathcal{I}_{\mathcal{N}}(\mathcal{H}(x))^\top D\mathcal{H}(x) (\mathcal{I}(x)^\top)^{\dagger}+ \mathcal{O}(\epsilon)  \nonumber
\end{align}
This implies that the regression based estimate of $D\mathcal{H}(x)$ can have large errors in directions orthogonal to the tangent space $T_x\mathcal{M}$.  However, these large errors are not important when $D\mathcal{H}(x)$ is used in constructing a local kernel, since the local kernel construction only depends on 
the projection of $D\mathcal{H}(x)$ onto the tangent space.  The likely reason that the linear regression, $YX^{\top}(XX^{\top})^{-1}$, gives better results than the correlation estimate, $\frac{1}{\epsilon}YX^\top$, is that the errors arising from the approximation of the continuous integrals by the finite summations in $YX^\top$ and $XX^\top$ are correlated, similar to the result found in \cite{SingerEstimate}.  Finally, we note that the columns in $Y$ require the value of $y = \mathcal{H}(x)$, which is assumed to be known in the training data set, but will not be known if we wish to extend the map $D\hat{\mathcal{H}}$ to a new point $x^*$.  However, this can easily be overcome by converting from a linear regression to an affine regression, which will implicitly estimate a weighted linear regression for $\mathcal{H}(x^*)$.

 \begin{figure}[h]
  \begin{center}
\includegraphics[width=.32\linewidth]{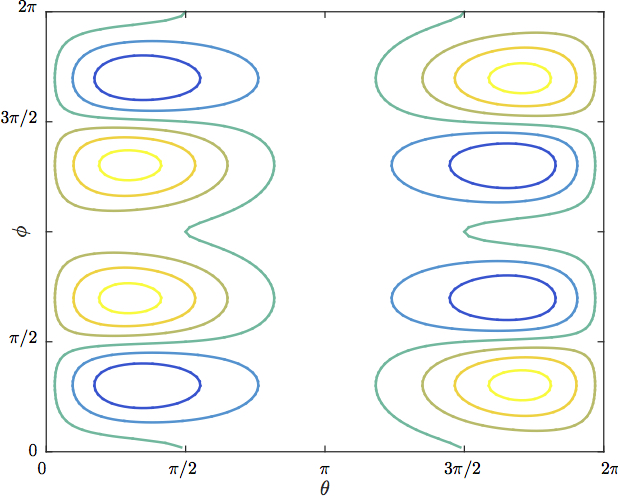}
 \includegraphics[width=.32\linewidth]{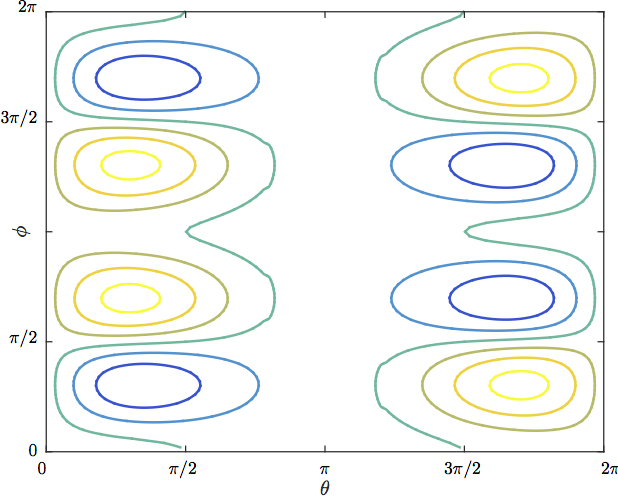}
  \includegraphics[width=.32\linewidth]{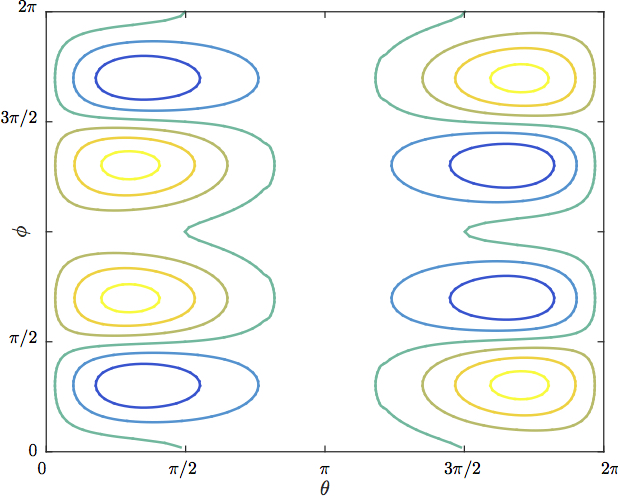} \\
  \includegraphics[width=.32\linewidth]{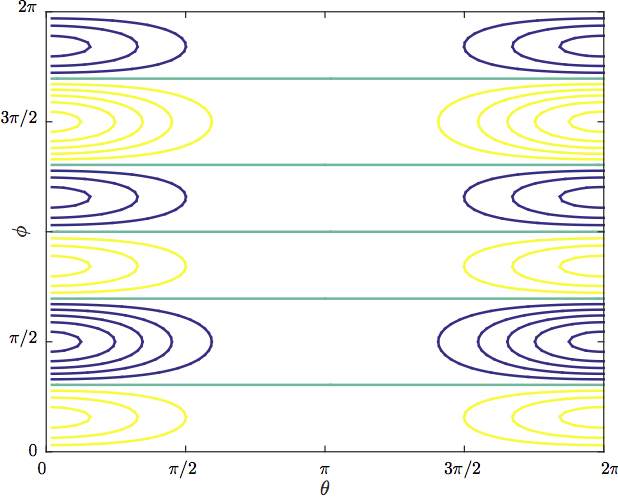}
 \includegraphics[width=.32\linewidth]{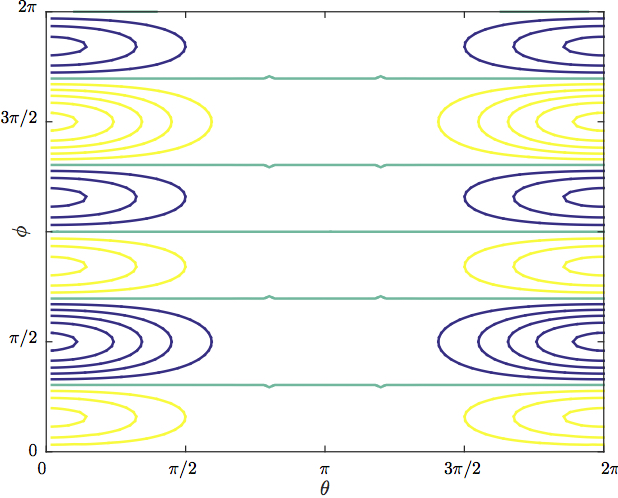}
  \includegraphics[width=.32\linewidth]{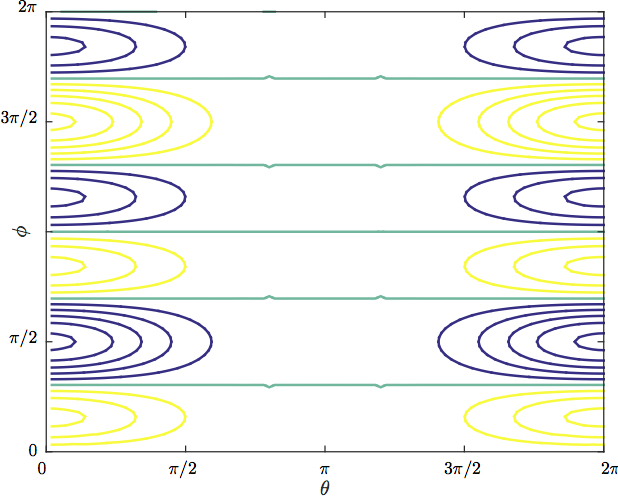} \\
  \includegraphics[width=.32\linewidth]{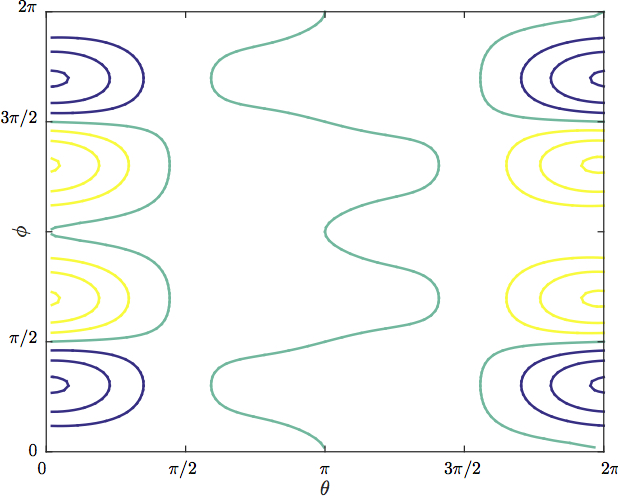}
 \includegraphics[width=.32\linewidth]{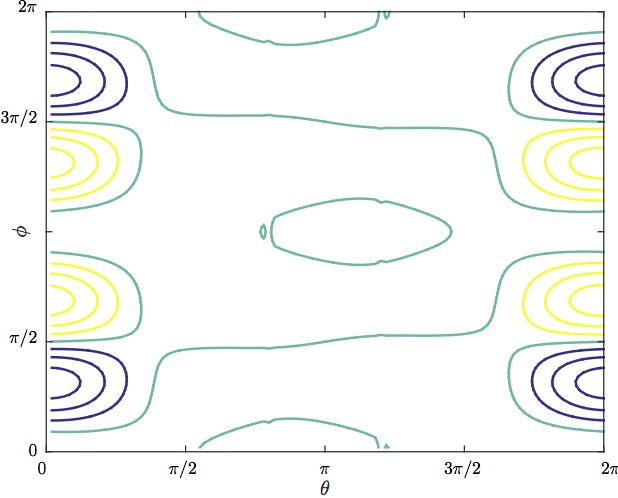}
  \includegraphics[width=.32\linewidth]{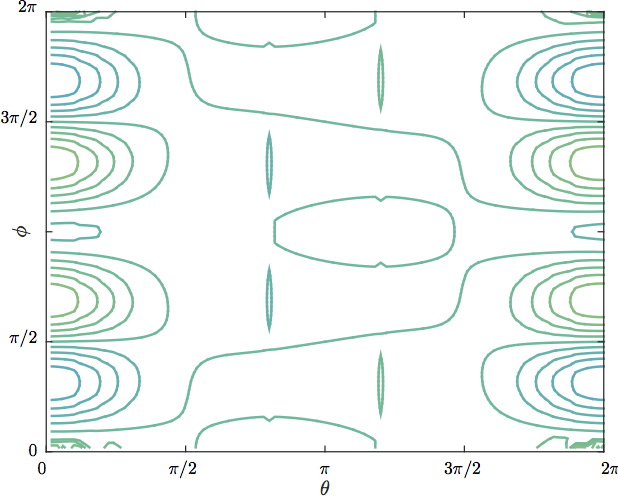}
\caption{\label{derivativeEstimates} Contour plots of derivatives $D\hat{\mathcal{H}}(x,y,z)\frac{d(x,y,z)}{d\theta}$ (top), $D\hat{\mathcal{H}}(x,y,z)\frac{d(x,y,z)}{d\phi}$ (middle), and $D\hat{\mathcal{H}}(x,y,z)\left(\frac{d(x,y,z)}{d\theta}\times \frac{d(x,y,z)}{d\phi}\right)$ (bottom) are shown for the analytical computation (first column), regression estimate (second column) and covariance estimate (third column). }
\end{center}
\end{figure}

We demonstrate this method of estimating derivatives by defining the function, 
\[ \hat{\mathcal{H}}(x,y,z) = xy^2 + z, \]
which restricted to the torus can be written in the coordinates $(\theta,\phi)$ as,
\[ \hat{\mathcal{H}}(x,y,z) = \mathcal{H}(\theta,\phi) = (2+\cos(\theta))^3\cos(\phi)\sin^2(\phi) + \sin(\theta). \] 
We evaluate $\hat{\mathcal{H}}$ on the data set lying exactly on the torus example in Section \ref{svd}. We will evaluate the derivative $D\hat{\mathcal{H}}(x,y,z) = (y^2,x,1)$ by projecting onto the two tangent directions $\frac{d(x,y,z)}{d\theta}$ and $\frac{d(x,y,z)}{d\phi}$ and the orthogonal direction $\frac{d(x,y,z)}{d\theta} \times \frac{d(x,y,z)}{d\phi}$.  In Figure~\ref{derivativeEstimates}, we compare the contour plot of the analytical derivatives (first column) to the corresponding estimates obtained by the linear regression (second column) and the covariance matrix (third column). Notice that the correlation matrix estimate $\frac{1}{\epsilon}YX^\top \approx D\hat{\mathcal{H}}$ is approximately zero when projected in the direction orthogonal to the tangent plane, whereas the linear regression estimate $YX^\top(XX^\top)^{-1}$ recovers the analytic derivative even in this orthogonal direction.  We re-emphasize that when used in a local kernel, the behavior in the orthogonal direction is irrelevant to the limiting operator.

\subsection{Tuning the Local Bandwidth via SVD}\label{tuning}

A significant challenge in applying kernel-based methods such as diffusion maps and local kernels is tuning the bandwidth parameter $\epsilon$.  The algorithms of \cite{diffusion,LK} are based on a global bandwidth parameter, meaning that the same value of $\epsilon$ is used for all data points.  In \cite{coifman2008TuningEpsilon} a method was introduced for tuning the global bandwidth parameter based on the scaling law in \eqref{sumscaling}.  As pointed out in \cite{coifman2008TuningEpsilon}, when $\epsilon$ is well chosen, the kernel $\exp\left(-\frac{||y - x ||^2}{2\epsilon}\right)$ will localize the integral over the whole manifold onto the tangent plane.  This localization is made rigorous up to an error of order-$\epsilon^{3/2}$ in Lemma 8 of \cite{diffusion}.  Thus, when $\epsilon$ is well-tuned we expect to see the scaling law $D(x) \propto \epsilon^{d/2}$.  On the other hand, in the limit as $\epsilon \to 0$ we find $D(x) \approx \frac{1}{N}\sum_{i=1}^N 0 = 0$ and in the limit as $\epsilon \to \infty$ we find $D(x) \approx \frac{1}{N}\sum_{i=1}^N 1 = 1$.  When using a global bandwidth, the approach advocated in \cite{coifman2008TuningEpsilon} was to average $D(x)$ over the dataset, and to choose bandwidth parameter $\epsilon$ so that 
\[ \overline{D}(\epsilon) = \frac{1}{N}\sum_{j=1}^N D(x_j) = \frac{1}{N^2}\sum_{i,j=1}^N \exp\left(-\frac{||y - x ||^2}{2\epsilon}\right) \propto \epsilon^{d/2}. \]  
Of course, this method of tuning the bandwidth parameter requires knowing the intrinsic dimension $d$ of the manifold $\mathcal{M}$.  In  \cite{coifman2008TuningEpsilon} they advocated choosing $\epsilon$ such that $\log(\overline{D}(\epsilon)) \approx \frac{d}{2}\log (\epsilon) + c$ is approximately linear as a function of $\log(\epsilon)$. 

In \cite{BH14VB}, an extension of the method of \cite{coifman2008TuningEpsilon} was advocated that simultaneously determines the bandwidth parameter $\epsilon$ and the intrinsic dimension $d$.  The approach of \cite{BH14VB} is based on the scaling law $S(\epsilon)$ defined by,
\[ S(\epsilon) \equiv \frac{d \log(\overline{D})}{d \log(\epsilon)}, \]   
and noting that when $\epsilon \to 0$ and $\epsilon \to \infty$ we have $S(\epsilon) \to 0$.  We should note that the limit $S(\epsilon)\to 0$ as $\epsilon\to 0$ applies only to the biased estimate $D(x_j)$ where the summation includes $i=j$, meaning that the largest summand is always $1$.  The largest summand being $1$ implies that the other summands will lose numerical significance as $\epsilon\to 0$, meaning $D$ converges to a constant and $S(\epsilon) \to 0$.  If the unbiased summation of $D(x_j)$ were used (for example in a Kernel Density Estimation) then as $\epsilon \to 0$ the summand corresponding to the shortest distance would dominate, so that $D \propto \exp(-c/\epsilon)$ and $S(\epsilon) =  \frac{d \log(D)}{d \log(\epsilon)} = \epsilon \frac{d (-c/\epsilon)}{d\epsilon} \propto \epsilon^{-1}$ in the limit as $\epsilon \to 0$. However, in this paper we restrict our attention to the biased estimate, as required by the diffusion maps and related algorithms, so that as $\epsilon\to 0$ we have $S(\epsilon)\to 0$. This implies that $S(\epsilon)$ has a unique maximum, and in \cite{BH14VB} they chose $\epsilon$ to maximize $S(\epsilon)$ and then set the dimension by, $d = 2S(\epsilon)$.  The approach of \cite{BH14VB} was found to be ineffective for kernels with a global bandwidth parameter, especially when there are large variations in the sizes of local neighborhoods due to the sampling of the data set.  However, the method of \cite{BH14VB} was found to be very robust for a variable bandwidth kernel of the form $\exp\left(-\frac{||y - x ||^2}{\epsilon \rho(x)\rho(y)}\right)$ where the bandwidth function $\rho(x)$ was chosen to be inversely proportional to a power of the sampling density, namely $\rho(x) \propto p(x)^{\beta}$ for $\beta<0$.  

From \eqref{sumscaling}, we should have a scaling law $D(x) \propto \epsilon^{d/2}$ in each local region.  We can now connect this fact to the scaling laws of the singular values shown above.  Recall that $X$ has $d$ singular values equal to $\sigma_l = \epsilon^{1/2} + \mathcal{O}(\epsilon)$, $l=1,...,d$ and the remaining $n-d$ singular values are order-$\epsilon$.  Thus, we have $\textup{trace}(XX^\top) = \sum_l \sigma_l = d\epsilon + \mathcal{O}(\epsilon^2)$ so that $\frac{1}{\epsilon}\textup{trace}(XX^\top) = d + \mathcal{O}(\epsilon)$. Since the trace is independent of the order of multiplication, we can define $\nu=(2\epsilon)^{-1}$ so that $d \log \nu = \frac{d\nu}{\nu} = - \frac{d\epsilon}{\epsilon} = -d\log \epsilon$  and write,
\begin{align} \frac{1}{\epsilon}\textup{trace}(XX^\top) &= 2\nu\,\textup{trace}(X^\top X) = \frac{2\nu}{D(x)} \sum_{i}  \exp\left(-\nu ||x_i-x||^2 \right) ||x_i - x||^2 = \frac{-2\nu}{D(x)}  \sum_{i}  \frac{d}{d\nu}\exp\left(-\nu ||x_i-x||^2\right) \nonumber \\
&=  \frac{-2\nu}{D(x)} \frac{d}{d\nu} \sum_{i}  \exp\left(-2\nu ||x_i-x||^2\right) = \frac{-2\nu}{D(x)} \frac{d D(x)}{d\nu} = 2 \frac{d \log D(x) }{d\log \epsilon }. \nonumber
\end{align}
The previous equation confirms that the scaling law of $D(x)$, given by,
\[ S_1(\epsilon) \equiv \frac{d\log(D(x))}{d\log(\epsilon)} \] 
should be equal to $d/2$, so one method of estimating the dimension for a given value of $\epsilon$ would be, 
\[ d_1(\epsilon) = 2 S_1(\epsilon), \]
and this formula uses the singular values by implicitly taking the trace of the matrix $XX^\top$.  This formula was previously known based on the fact that $D(x) \propto \epsilon^{d/2}$, which comes from the normalization factor for a Gaussian on $T_x\mathcal{M}$.  However, the connection to the sum of the singular values reveals that when the ambient space dimension, $m$, is large, the singular values $\sigma_l$ for $l>d$ can lead to overestimation since, 
\[ d_1(\epsilon) = 2 S_1(\epsilon) = \frac{1}{\epsilon}\textup{trace}(XX^\top) = d + \sum_{l=d+1}^m \sigma_l /\epsilon . \]
Of course, each $\sigma_l$ is order-$\epsilon^2$ for $l>d$, however, when $m$ is large enough, this summation can lead to significant overestimation.  We note that the coefficients of these order-$\epsilon^2$ singular values depend on the curvature of the manifold at the point $x$, and these coefficients can be large for complex geometries.  This shows how the value of $\epsilon$, which maximizes the local scaling law $S_1(\epsilon)$, as suggested in \cite{BH14VB}, can overestimate the dimension.  

 \begin{figure}[h]
  \begin{center}
\includegraphics[width=.45\linewidth]{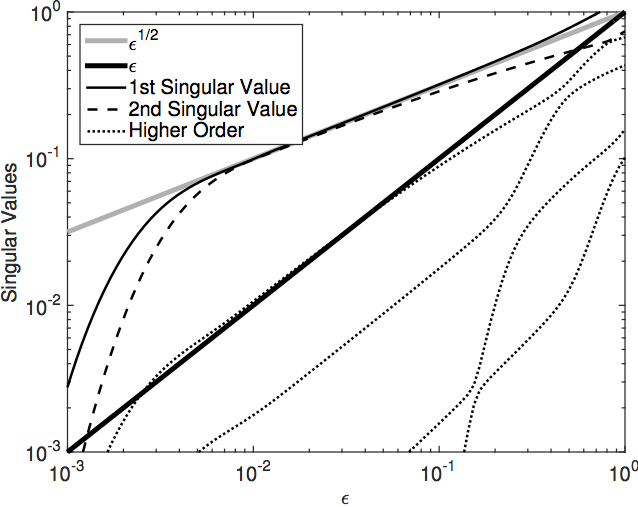}
 \includegraphics[width=.45\linewidth]{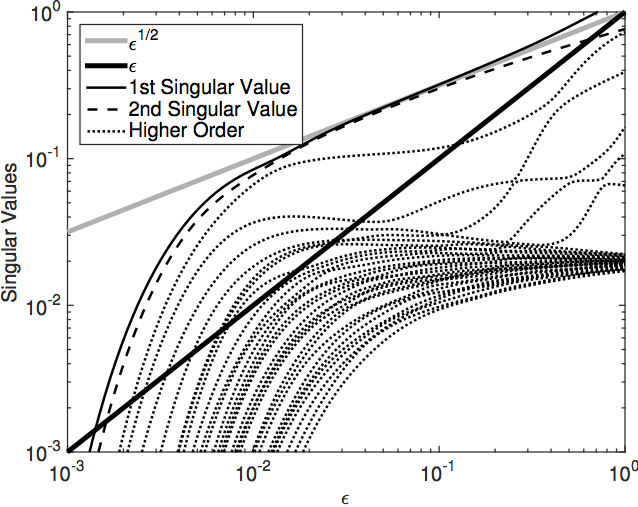}
\caption{\label{singularVals} Singular values as a function of $\epsilon$ for a high curvature embedding of a torus into $\mathbb{R}^{30}$ (left) and the same data set perturbed by 30-dimensional additive Gaussian noise with mean zero and covariance matrix $\frac{1}{50}I_{30\times 30}$ (right).}
\end{center}
\end{figure}

Here, we introduce a new method that combines the ideas of \cite{coifman2008TuningEpsilon,BH14VB} with the local SVD in order to tune $\epsilon$ in each local region and improve approximation of the tangent space.  Recall from Section \ref{svd}, in a local region of $x\in\mathcal{M}$, we define the matrix of weighted vectors, $X$, with columns,
\[ X_i = D(x)^{-1/2}\exp\left(-\frac{||x_i - x ||^2}{4\epsilon}\right)(x_i-x). \]  
Letting $\sigma_l$ be the singular values of $X$, when $\epsilon$ is well tuned the first $d$ singular values obey the scaling law $\sigma_l \propto \sqrt{\epsilon}$ and the remaining $m-d$ singular values (where $m$ is the ambient space dimension) are higher order, namely $\sigma_l = \mathcal{O}(\epsilon)$.  Notice that the $m-d$ singular values which are $\mathcal{O}(\epsilon)$ are not necessarily proportional to $\epsilon$; indeed they can be exactly zero in the case of a linear manifold such as a plane embedded in $\mathbb{R}^3$.  One strategy would be to threshold the singular values, however, by adding a small amount of noise to the data set in the ambient space, we can easily produce singular values which are greater than $\epsilon$.  We illustrate these issues in Figure \ref{singularVals} by embedding a torus into $\mathbb{R}^{30}$ where the first three coordinates are the standard embedding of the torus and the remaining 27 coordinates results from applying a randomly-generated orthogonal transformation to the first three coordinates raised to the third power and divided by 30.  Cubing the coordinates results in a high curvature embedding, which leads to large constants on the $\mathcal{O}(\epsilon)$ bound on the singular values corresponding to singular vectors that are orthogonal to the manifold.  The orthogonal transformation generates a nontrivial embedding into $\mathbb{R}^{30}$ and the addition of Gaussian noise makes this a highly complex embedding of an intrinsically simple data set.  In Figure \ref{singularVals} we show the singular values for the clean and noisy 30-dimensional embeddings. Notice that thresholding singular values less than $\epsilon$ may be effective when the data lies exactly on the manifold (left), however the high curvature can result in nontrivial constants in the $\mathcal{O}(\epsilon)$ bound.  The addition of noise implies that the dimension of the manifold is greater than two for some values of $\epsilon$ (for example $\epsilon \approx 10^{-2}$).  For $\epsilon \in [2 \times 10^{-2},10^{-1}]$ the third largest singular value is larger than $\epsilon$ but does not obey the scaling law $\epsilon^{1/2}$.  While thresholding alone cannot detect the two-dimensional structure, the scaling laws reveal the true dimension of the manifold.

To incorporate the scaling laws of the singular values into the tuning of $\epsilon$ and the dimension estimation, we introduce the following measure of dimension,
\[ d_2(\epsilon) \equiv 2 \sum_{l=1}^{\textup{floor}(d_1)} \frac{d \log(\sigma_l)}{d \log(\epsilon)} + 2 (d_1-\textup{floor}(d_1)) \frac{d \log(\sigma_{\textup{floor}(d_1)+1})}{d \log(\epsilon)}. \] 
Notice that when $d_1$ is an integer, the second term is zero, and the summation is simply the sum of the first $d_1$ scaling laws.  If the first $d_1$ singular values correspond to tangent vectors, then the associated scaling laws should be $1/2$, and in this case, we would find $d_2 = 2 \sum_{l=1}^{d_1} 1/2 = d_1$.  More generally, we can see that the summation can be rewritten as, 
\[ 2 \sum_{l=1}^{\textup{floor}(d_1)} \frac{d \log(\sigma_l)}{d \log(\epsilon)} = 2 \frac{d}{d\log(\epsilon)} \log \left( \prod_{l=1}^{d_1} \sigma_l \right), \]
which reveals this second dimension to be related to the determinant since it comes from a product of singular values (as opposed to $d_1$, which comes from a summation of singular values).  The final term is included so that $d_2$ is a smooth function of $\epsilon$. 

 \begin{figure}[h]
  \begin{center}
\includegraphics[width=.45\linewidth]{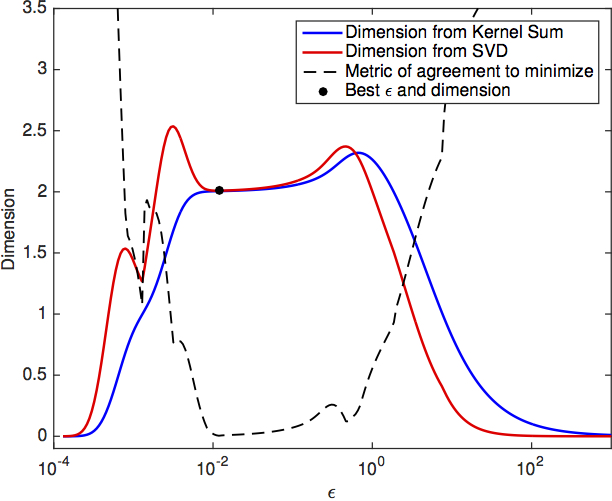}
 \includegraphics[width=.45\linewidth]{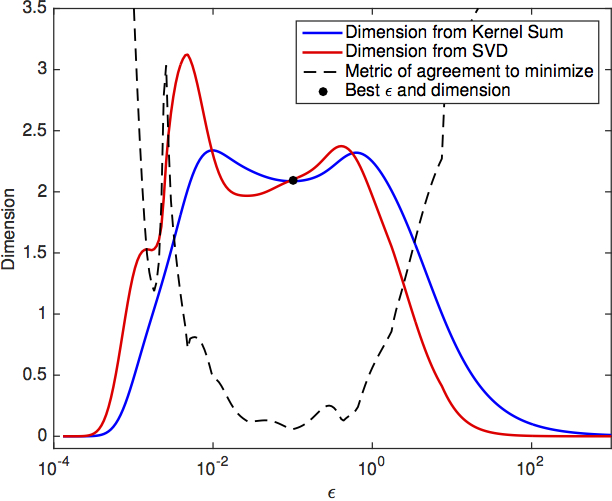} 
 \includegraphics[width=.45\linewidth]{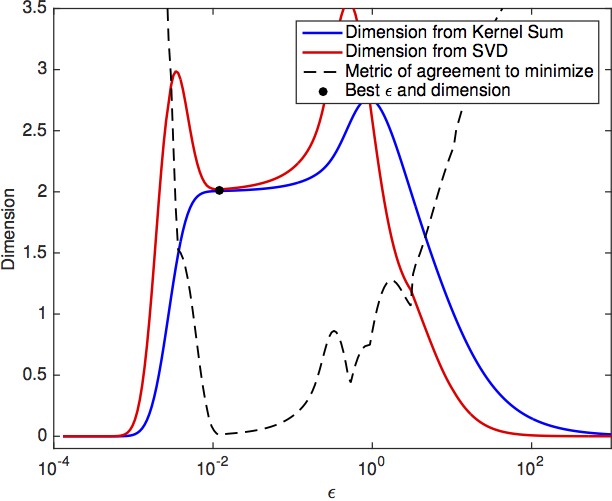}
 \includegraphics[width=.45\linewidth]{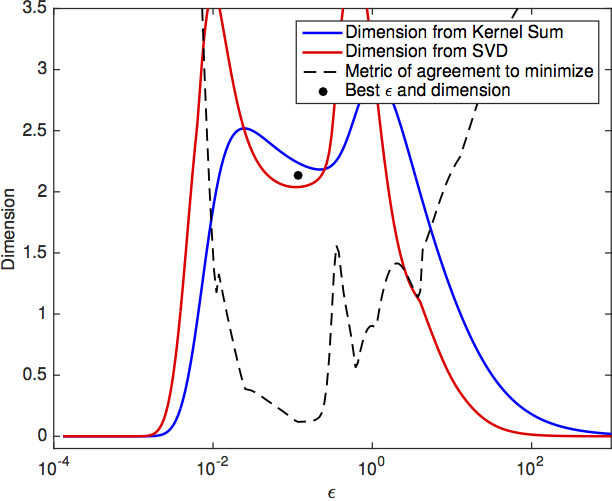} 
\caption{\label{tuningEpsilon} Top Row: Dimension measures $d_1$ (blue) and $d_2$ (red) as functions of the bandwidth $\epsilon$ at the base point $(1.996,0.126,1.000)^\top$ corresponding to the data set sampled from the torus (left) and the noisy torus (right) shown in Figure \ref{svdscaling} of Section \ref{svd}.  The metric of agreement, $M(\epsilon)$, is shown as the dotted black curve.  The solid black dot represents the bandwidth that minimizes the metric along with the average dimension at the optimal $\epsilon$. Bottom Row: Same curves for the 30-dimensional high-curvature embedding used in Figure \ref{singularVals} (left) and with 30-dimensional Gaussian noise (right)}
\end{center}
\end{figure} 

For each value of $\epsilon$ we now have two estimates the dimension, and when $\epsilon$ is well-tuned these two estimates of the intrinsic dimension should agree, so we choose $\epsilon$ to minimize the relative disagreement $\frac{d_1(\epsilon) - d_2(\epsilon)}{d_{\textup{ave}}(\epsilon)}$ where we set the intrinsic dimension to be, 
\[ d_{\textup{ave}}(\epsilon) \equiv (d_1(\epsilon) + d_2(\epsilon))/2. \]
A slight complication is that the curves $d_1(\epsilon)$ and $d_2(\epsilon)$ can intersect multiple times, as shown in Figure \ref{tuningEpsilon}.  In order to ensure that the scaling laws are stationary at the intersection point, we would also like to minimize the derivatives $\left|\frac{d\log d_1}{d\log \epsilon}\right|$ and $\left|\frac{d\log d_2}{d\log \epsilon}\right|$.  Thus, as a practical method of choosing $\epsilon$, we minimize the metric,
\[ M(\epsilon) \equiv  \left|\frac{d_1(\epsilon)-d_2(\epsilon)}{d_{\textup{ave}}(\epsilon)}\right| + \left|\frac{d\log d_1}{d\log \epsilon}\right| + \left|\frac{d\log d_2}{d\log \epsilon}\right| \]
where the derivatives are numerically discretized.

We demonstrate this method of tuning the bandwidth $\epsilon$ on the example in Section \ref{svd} and the results are shown in the top row of Figure \ref{tuningEpsilon}. We also applied this method of tuning the bandwidth to the 30-dimensional high curvature embedding from Figure \ref{singularVals}, and the results are shown in the bottom row of Figure \ref{tuningEpsilon}. The optimal bandwidth shown in the top row of Figure \ref{tuningEpsilon} was used to plot the singular vectors in Figure \ref{svdscaling} above. 

We should note that there are many other approaches one could take to estimate the intrinsic dimension of manifolds using the facts introduced in this section. In particular, there are many thresholding methods that could be applied to find integer dimensions.  Motivated by applications to noisy and fractal data sets (which fall outside of the current theory) we have developed a non-integer measure of dimension based on scaling laws.  Moreover, notice that in Figure \ref{annulus1}, different parts of the manifold contract at different rates, so that the dimension of the manifold does not appear constant.  As a result of this, in Section \ref{section4} and in \ref{numerics} we will use the rescaled diffusion mapping $\hat{\Phi}$ of Section \ref{section21} with a locally determined dimension $d(x_i)$.  Whichever method is used to estimate dimensions, the examples in this section show that both the magnitudes and the scaling laws of the singular values should be incorporated.

A significant drawback of the method of tuning the bandwidth $\epsilon$, introduced in this section, is that computing $d_2(\epsilon)$ requires computing the singular value decomposition of the weighted vectors $X$, for every base point and a large range of bandwidth parameters.  Due to the increase in computational complexity, in all of the examples below (and in the algorithm of \ref{numerics}) we use the simple method of maximizing $d_1$ to choose the bandwidth.  We suspect that this method of choosing the bandwidth is sufficient for the examples below due to low curvature embeddings with small noise, so we included this new method of tuning $\epsilon$ to demonstrate a robust tuning method for more complex data sets.

\section{Iterated Diffusion Map (IDM)}\label{section4}

In this section we consider representing general maps $\mathcal{H}$ that can take data in high-dimensional spaces to lower-dimensional spaces, generalizing the result in \cite{LK} that was reviewed in Section~\ref{section22}. In particular, we will make use of Theorem~\ref{maintheorem} to find an isometric embedding of $\mathcal{M}$ with respect to the appropriate geometry such that these new embedded coordinates emphasize the feature of interest $\mathcal{H}(\mathcal{M})=\mathcal{N}$. In analogy to the diagram in Section~\ref{section22}, we shall see that the proposed method represents $\mathcal{H}$ with a linear map between the iterated diffusion mapping of the data manifold $\mathcal{M}$ and the rescaled diffusion coordinates of feature space $\mathcal{N}$.  

One of the challenges is that the result in \cite{LK} is not immediately applicable since $\mathcal{H}$ is not assumed to be a diffeomorphism, and therefore the kernel constructed in \eqref{diffeokernel} from $D\mathcal{H}$ is not necessarily a local kernel. To see this, we can define a covariance matrix $C(x)^{-1} = D\hat{\mathcal{H}}(x)^\top D\hat{\mathcal{H}}(x)$, where $D\hat{\mathcal{H}}(x)$ is the local derivative in the ambient space estimated by linear regression as discussed in Section~\ref{linreg}. If we naively form the kernel $K(\epsilon,x,y)$ from \eqref{PK} with covariance matrix $C(x)$, then this will not be a local kernel. The problem is that the restriction of $C(x)^{-1}$ to the tangent plane, $c(x)^{-1} = \mathcal{I}(x)C(x)^{-1}\mathcal{I}(x)^\top = D{\mathcal{H}}(x)^\top D{\mathcal{H}}(x)$, may not be full rank since the map $\mathcal{H}$ may take the manifold $\mathcal{M}$ to a lower-dimensional manifold $\mathcal{H}(\mathcal{M})$. If $c(x)^{-1}$ is not full rank, then there exists a nontrivial vector $u \in T_x\mathcal{M}$ such that $u^\top c(x)^{-1} u = 0$ (in fact $c(x)^{-1}u=0$), so if $y-x = (u,q(u))$ we find $K(\epsilon,x,y) = \mathcal{O}(1)$, which means that $K$ does not have the exponential decay, so $K$ is not a local kernel (see Section~\ref{section22} and \cite{LK}).

Often the kernel $K$ is constructed using the $k$ nearest neighbors, so that $K(\epsilon,x,y) \equiv 0$ by definition when $y$ is not in the list of the $k$ nearest neighbors of $x$, and vice-versa.  When the $k$ nearest neighbor algorithm is used, technically the kernel $K$ constructed with a rank deficient covariance matrix is still a local kernel since the kernel still has an implicit decay that can be bounded above by an exponential function.  However, the localization caused by the $k$ nearest neighbor algorithm has a very sharp cutoff such that the corresponding operator approximated by the kernel is very sensitive to the choice of $k$. 

In order to use the local kernels theory to represent the feature map $\mathcal{H}$, we propose a novel algorithm called the iterated diffusion map (IDM).  The IDM will make use of local kernels which use small perturbations of identity covariance matrices such that Theorem~\ref{maintheorem} is applicable on each iteration. In Section~\ref{IDM}, we present the IDM and show that it is a discrete approximation of an intrinsic geometric flow. In Section~\ref{productman}, we show that if the data space $\mathcal{M}$ is a product of the feature space and the irrelevant space, then IDM will produce a quotient manifold that is isometric to the feature space, eliminating the irrelevant dimension. Finally, we will show numerical results with IDM in Section~\ref{examples}, highlighting its advantages and limitations.  The numerical algorithm of the IDM is outlined in \ref{numerics}.

\subsection{IDM as an Intrinsic Geometric Flow}\label{IDM}

We now introduce the IDM algorithm for feature identification. The method assumes the availability of a pair of data sets $x_i\in\mathcal{M}\subset\mathbb{R}^m$ and $y_i=\mathcal{H}(x_i)\in\mathcal{N}\subset\mathbb{R}^n$, where $\mathcal{H}$ is not assumed to be a diffeomorphism and $\mathcal{N}$ may even be lower dimension than $\mathcal{M}$. With this training data, we apply the linear regression method in Section~\ref{linreg} to approximate the local derivative $D\hat{\cal H}$ in the ambient space which is subsequently used to define a new covariance,
\begin{align}\label{Cdel} 
C_{\mathcal{H}^{(0)}}(x) = \left((1-\tau)\textup{I}_{m\times m} +  \tau D\hat{\mathcal{H}}(x)^\top D\hat{\mathcal{H}}(x)\right)^{-1}, 
\end{align}
where $\textup{I}_{m\times m}$ is the $m\times m$ identity matrix. Notice that with this construction, $C_{\mathcal{H}^{(0)}}(x)$ is guaranteed to be positive definite, even when $D{\mathcal{H}}(x)^\top D{\mathcal{H}}(x)$ is not a full rank matrix (where the relation of $D\mathcal{H}$ and $D\hat{\cal H}$ is defined in \eqref{DhatH}). With the definition in \eqref{Cdel}, we implicitly define a map $\mathcal{G}:\mathcal{M}\to\mathcal{M}$ such that 
$D\mathcal{G}(x)^\top D\mathcal{G}(x)= C_{\mathcal{H}^{(0)}}(x)^{-1}$. When $\tau\ll 1$, intuitively, $\mathcal{G}$ is a small perturbation of an identity map on $\mathcal{M}$ since
\begin{align}
D\mathcal{G}(x) = \textup{I}_{m\times m} - \frac{1}{2}\tau \Big(D\hat{\mathcal{H}}(x)^\top D\hat{\mathcal{H}}(x)-\textup{I}_{m\times m}\Big)+\mathcal{O}(\tau^2).\nonumber
\end{align}  
Unlike the sharp decay due to the $k$ nearest neighbor cutoff, $C_{\mathcal{H}^{(0)}}(x)$ achieves a smooth decay even in directions where $D\mathcal{H}(x)D\mathcal{H}(x)^\top$ is rank deficient. Using the prototypical kernel, \[K(\epsilon,x,y) = \exp\left(-(y-x)^\top C_{\mathcal{H}^{(0)}}(x)^{-1}(y-x)/2 \right),\] along with the construction in Theorem \ref{maintheorem}, we approximate the operator $\Delta_{g_{\mathcal{H}^{(0)}}}$ which is the Laplace-Beltrami operator with respect to the Riemannian metric, 
\begin{align}\label{gH} 
g_{\mathcal{H}^{(0)}} = c_{\mathcal{H}^{(0)}}^{-1/2}g_{\cal M} c_{\mathcal{H}^{(0)}}^{-1/2} = ((1-\tau)\textup{I}_{d\times d} + \tau D{\mathcal{H}}^\top D{\mathcal{H}})^{1/2} g_{\cal M} ((1-\tau)\textup{I}_{d\times d} +  D{\mathcal{H}}^\top D{\mathcal{H}})^{1/2}, \end{align}
where $c_{\mathcal{H}^{(0)}}(x) = \mathcal{I}(x)C_{\mathcal{H}^{(0)}}(x)\mathcal{I}(x)^\top$. 
Notice that if we build a diffusion map $\Phi^{(0)}_{s}(x) =
 (e^{s\lambda_1}\varphi_1(x),...,e^{s\lambda_M}\varphi_M(x))^\top\equiv x^{(1)}$ using the eigenfunctions of $\Delta_{g_{\mathcal{H}^{(0)}}}$ (approximated by the local kernel construction) this gives an approximately isometric embedding of $\mathcal{M}$ with respect to the metric $g_{\mathcal{H}^{(0)}}$, for small enough parameter $s$. Moreover, the new metric $g_{\mathcal{H}^{(0)}}$ in \eqref{gH} puts a larger weight on directions in which $D{\mathcal{H}}$ are large, which are the direction associated with the range space of $\mathcal{H}$. 

The key point that makes the iterated diffusion map useful is that the local kernels with covariance defined below (cf. \eqref{covariance}), change the geometry, as opposed to iterating the diffusion maps using identity covariance, $C(x)=I_{m\times m}$, as discussed in Section~\ref{section21}. In particular, the $\ell$-th iteration is performed on the coordinate 
\begin{align}
x^{(\ell-1)}=\Phi^{(\ell-2)}_{s}(x^{(\ell-2)}),  \quad \ell = 2,3,\dots, 
\end{align}
where $x^{(0)}\equiv x$, with induced feature maps $\mathcal{H}^{(\ell-1)}: \mathbb{R}^M\to\mathcal{N}$ defined as follows, 
\begin{align} \mathcal{H}^{(\ell-1)}(x^{(\ell-1)}) \equiv \mathcal{H}(x), \quad   \ell = 2,3,\dots. \label{featuremap}\end{align}
Numerically, we approximate the local derivative $D\hat{\mathcal{H}}^{(\ell-1)}$ of $\mathcal{H}^{(\ell-1)}$ in the ambient space by the linear regression method in Section~\ref{linreg}. In this particular implementation, $X_j^{(\ell-1)}$ and $Y_j^{(\ell-1)}$ in Theorem~\ref{thm1} are defined as,
\begin{align}
X_j^{(\ell-1)} &= D(x^{(\ell-1)})^{-1/2} \exp{\left( -\frac{\|x_j^{(\ell-1)}-x^{(\ell-1)}\|^2}{4\epsilon}\right)}(x_j^{(\ell-1)}-x^{(\ell-1)}), \nonumber\\ Y_j^{(\ell-1)} &= D(x^{(\ell-1)})^{-1/2} \exp{\left( -\frac{\|x_j^{(\ell-1)}-x^{(\ell-1)}\|^2}{4\epsilon}\right)}(y_j-y), \nonumber
\end{align}
where $x_j^{(\ell-1)}:= (\Phi^{(\ell-2)}_{s}\circ  \Phi^{(\ell-3)}_{s} \circ \ldots \circ  \Phi^{(0)}_{s}) (x_j)$ for $\ell\geq 2$. Given $D\hat{\mathcal{H}}^{(\ell-1)}$, we define local kernels induced by covariance matrices, 
\begin{align}C_{\mathcal{H}^{(\ell-1)}}(x^{(\ell-1)}) = \left((1-\tau)\textup{I}_{M\times M} +\tau  D\hat{\mathcal{H}}^{(\ell-1)}(x^{(\ell-1)})^\top D\hat{\mathcal{H}}^{(\ell-1)}(x^{(\ell-1)})\right)^{-1}, \quad \ell= 2, 3, \ldots.\label{covariance} \end{align}
We can now repeat the local kernel construction above using the covariance $C_{\mathcal{H}^{(\ell-1)}}(x^{(\ell-1)})$ to produce eigenfunctions and eigenvalues of $\Delta_{g_\mathcal{H}^{(\ell-1)}}$ and obtain $x^{(\ell)}= \Phi^{(\ell-1)}_{s}(x^{(\ell-1)})$.
  For $s$ sufficiently small, the new coordinates $x^{(\ell)}$ will be an approximately isometric embedding of $\mathcal{M}$ with respect to the metric,
 \begin{align}
 g_{\mathcal{H}^{(\ell-1)}} =  c_{\mathcal{H}^{(\ell-1)}}^{-1/2}g_{\mathcal{H}^{(\ell-1)}} c_{\mathcal{H}^{(\ell-1)}}^{-1/2} 
 = c_{\mathcal{H}^{(\ell-1)}}^{-1/2} \cdots c_{\mathcal{H}^{(0)}}^{-1/2}g_{\cal M}\, c_{\mathcal{H}^{(0)}}^{-1/2} \cdots c_{\mathcal{H}^{(\ell-1)}}^{-1/2},\label{metricmap}
 \end{align}
where $c_{\mathcal{H}^{(j)}}(x) = \mathcal{I}(x)C_{\mathcal{H}^{(j)}}(x)\mathcal{I}(x)^\top$ for $j=0,\ldots,\ell-1$.
Each iteration of the diffusion map further emphasizes the directions on the manifold $\mathcal{M}$, which are important to the function $\mathcal{H}$. Moreover, the map $\mathcal{H}(x)$ is a fixed point of the iterated diffusion map process.  To see this, assume that for some $k$ we have $x^{(k)} = \mathcal{H}(x)$, then we find
\[ \mathcal{H}^{(k)}(\mathcal{H}(x)) = \mathcal{H}^{(k)}(x^{(k}) = \mathcal{H}(x), \]
so $D\mathcal{H}^{(k)}D\mathcal{H} = D\mathcal{H}$ which implies that $D\mathcal{H}^{(k)} = I_{d\times d}$ such that $c_{\mathcal{H}^{(k)}}=I_{d\times d}$ so $g_{\mathcal{H}^{(k+1)}} = g_{\mathcal{H}^{(k)}}$ and therefore $x^{(k+1)} = x^{(k)} = \mathcal{H}(x)$.  Similarly, any isometric embedding $\iota_{\mathcal{N}}(\mathcal{H}(x))$ is a fixed point of the iterated diffusion map.  To see this, assume $x^{(k)} = \iota_{\mathcal{N}}(\mathcal{H}(x))$ and note that $\mathcal{H}^{(k)}(\iota_{\mathcal{N}}(\mathcal{H}(x))) = \mathcal{H}(x)$ so that $D\mathcal{H}^{(k)}D\iota_{\mathcal{N}}D\mathcal{H} = D\mathcal{H}$. Since $\iota_{\mathcal{N}}$ is an isometric embedding, we have $D\iota_{\mathcal{N}}=I$ and this implies that $D\mathcal{H}^{(k)} = I$. It remains an open question whether this fixed point is attracting in the general case, and further analysis is needed to understand this issue.

One possible interpretation of the IDM is as a discretization of a geometric flow. To see this, we define,
\[ c(x,t+\tau) = \left((1-\tau)\textup{I}_{d\times d} +  \tau D\mathcal{H}(x(t))^\top D\mathcal{H}(x(t))\right)^{-1}, \]
as a continuous analog of \eqref{covariance}, where $x(t)$ is an isometric embedding of $(\mathcal{M},g(t))$ where $g(0)=g_{\mathcal{M}}$ and with a feature map defined continuously $\mathcal{H}(x(t)) = \mathcal{H}(x(0))$, where $x(0) = x$ to mimic the discrete setting in \eqref{featuremap}. The new metric introduced by $c(x,t+\tau)$ would be,
\begin{align} 
g(t+\tau) &= c(x,t+\tau)^{-1/2}g(t) c(x,t+\tau)^{-1/2}\nonumber\\ 
&= (\textup{I} + \tau (D\mathcal{H}^\top D\mathcal{H}-\textup{I}))^{1/2} g(t) (\textup{I} + \tau (D\mathcal{H}^\top D\mathcal{H}-\textup{I}))^{1/2} \nonumber \\ 
&= g(t) + \frac{\tau}{2} \Big(D\mathcal{H}^\top D\mathcal{H} g(t) + g(t)D\mathcal{H}^\top D\mathcal{H}-2g(t)\Big) + \mathcal{O}(\tau^2).  
\end{align}
Rewriting the previous equation we find,
\begin{align} \frac{dg}{dt} = \lim_{\tau\rightarrow 0} \frac{g(t+\tau) - g(t)}{\tau} =-g+ \frac{1}{2} \Big(D\mathcal{H}^\top D\mathcal{H} g + gD\mathcal{H}^\top D\mathcal{H}\Big), \label{flow}\end{align}
which is an equation describing an intrinsic geometric flow. Notice again that if $D\mathcal{H}=I$, then $g$ is an equilibrium solution of \eqref{flow}. We should note that the geometric flow in \eqref{flow} is nonlinear since the map $\mathcal{H}$ depends on the metric $g$ in nontrivial fashion (since $\mathcal{H}$ maps an isometric embedding of $(\mathcal{M},g(t))$ to the feature of interest in $\mathcal{H}(\mathcal{M})$). This is the reason why it is not straightforward to see whether there are other equilibrium solutions or even to determine the stability of any equilibrium solution. The standard linear stability analysis suggests that if $g^*$ is the fixed point of \eqref{flow}, then $g^*$ is locally attracting when the real part of all of the eigenvalues of the linearized operator $D_g\Big(D\mathcal{H}^\top D\mathcal{H} g + gD\mathcal{H}^\top D\mathcal{H}\Big)|_{g=g^*}$ is less than 2. 

\subsection{IDM for Product Manifolds}\label{productman}

We now consider a simple case where the manifold $\mathcal{M}$ is a product space $\mathcal{M} = \mathcal{N} \times \mathcal{P}$, such that $\mathcal{N}$ is the feature space and $\mathcal{P}$ contains variables we wish to ignore.  In this case, the map $\mathcal{H}:\mathcal{M} \to \mathcal{N}$ has a particularly simple structure.  In each local neighborhood we can find coordinates $x = (y,z) \in \mathcal{M}$ where $y$ are coordinates on $\mathcal{N}$ and $z$ are coordinates on $\mathcal{P}$.  In these coordinates the metric $g$ will naturally decompose into a block diagonal matrix.  The first block represents the metric $g_{\mathcal{N}}$ on $\mathcal{N}$ and this block is $d_{\mathcal{N}} \times d_{\mathcal{N}}$ and the second block represents the metric $g_{\mathcal{P}}$ on $\mathcal{P}$ and this block is $d_{\mathcal{P}} \times d_{\mathcal{P}}$. Since $\mathcal{H}(\mathcal{M}) = \mathcal{N}$ maps each point to the feature of interest, in these local coordinates, the feature map will take the form $\mathcal{H}(x) = \mathcal{H}(y,z) = y$.  Moreover, in these coordinates, $D{\mathcal{H}}(x)$ is a block diagonal matrix where the first $d_{\mathcal{N}} \times d_{\mathcal{N}}$ submatrix is the identity matrix and the remaining entries are all zero.  So we find that $D{\mathcal{H}}^\top D{\mathcal{H}}-I$ is again a block diagonal matrix, where the bottom $d_{\mathcal{P}}\times d_{\mathcal{P}}$ block is equal to $-I$ and the remaining entries are zero.  Writing the geometric flow \eqref{flow} in these coordinates we find,
\begin{align} \dot g = \frac{1}{2} \left((D{\mathcal{H}}^\top D{\mathcal{H}}-I) g + g(D{\mathcal{H}}^\top D{\mathcal{H}} - I)\right) = \frac{1}{2}\left(\left(\begin{array}{cc} 0 & 0 \\ 0 & -I \end{array}\right)\left( \begin{array}{cc} g_{\mathcal{N}} & 0 \\ 0 & g_{\mathcal{P}} \end{array}\right) + \left( \begin{array}{cc} g_{\mathcal{N}} & 0 \\ 0 & g_{\mathcal{P}} \end{array}\right)\left(\begin{array}{cc} 0 & 0 \\ 0 & -I \end{array}\right) \right), \label{prodflow}\end{align}
which implies that $\dot g_{\mathcal{N}} = 0$ and $\dot g_{\mathcal{P}} = -g_{\mathcal{P}}$.  This shows that for product manifolds of the geometric flow \eqref{flow} will contract the irrelevant variables to zero and leave the features of interest unchanged.  So for sufficiently small discretization $\tau$ and in the limit of sufficiently many iterations, the IDM will construct the quotient map from the product manifold to an isometric copy of the feature space.  At this point, the data can easily be mapped to the feature space using the method of Section \ref{section22}.  In fact, since the quotient manifold is already isometric to the feature space, one could simply estimate a linear map between the rescaled diffusion coordinates of the quotient manifold and those of the feature space (since these coordinates are canonical up to rotation as shown in Section \ref{section21}).  In analogy to the diagram in Section \ref{section22} which represents a diffeomorphism, we can summarize the IDM construction of the quotient map with the following diagram,
  \begin{center}\vspace{-15pt}
$$  \begin{array}[c]{ccc}
\mathcal{M}=\mathcal{N}\times \mathcal{P}&\xrightarrow{\ \ \ \ \ \mathcal{H}\ \ \ \ \ }&\mathcal{N}\\ \\
\hspace{50pt}\left\downarrow\rule{0cm}{.5cm}\right.   \scriptstyle{\Psi \equiv \lim_{\ell \to \infty,s\to 0}\Phi^{(\ell)}_{s}\circ \cdots \circ  \Phi^{(0)}_{s}}&&\left\downarrow\rule{0cm}{.5cm}\right. \scriptstyle{\tilde\Phi}\\ \\
L^2(\mathcal{N},\tilde g) \approx \mathbb{R}^{M}&\xrightarrow{\ \ \ \ \ H\ \ \ \ \ }&L^2
(\mathcal{N},g_{\mathcal{N}}) \approx \mathbb{R}^{M}
\end{array}$$
\end{center}
where $\Psi$ represents the iterated diffusion map, $\tilde \Phi$ are the rescaled diffusion coordinates of $\mathcal{N}$.  The above diagram shows how $\mathcal{H}$ is represented by an orthogonal linear transformation $H$ via $\mathcal{H} = \tilde\Phi^{-1} \circ H \circ \Psi$.  

Moreover, consider the case when $\mathcal{M} = \mathcal{N}\times\mathcal{P}$, but the feature of interest is $\mathcal{F}(\mathcal{N})$, where $\mathcal{F}$ is a diffeomorphism.  In this case, the block diagonal structure $D\mathcal{H}$ and of \eqref{prodflow} will still hold, and in particular we still find $\dot g_{\mathcal{P}} = -g_{\mathcal{P}}$.  This shows that the flow still contracts the irrelevant variables $\mathcal{P}$ to zero, and the only difference is that we will find $\dot g_{\mathcal{N}} = \frac{1}{2}\left((D\mathcal{F}^\top D\mathcal{F}-I)g_{\mathcal{N}} + g_{\mathcal{N}}(D\mathcal{F}^\top D\mathcal{F}-I)\right)$. Notice that the fixed point for this flow 
satisfies $D\mathcal{F}=I$, so we expect in the limit to obtain an isometric copy of $\mathcal{N}$.  However, even if the flow on $g_{\mathcal{N}}$ has not converged, once the IDM has contracted the irrelevant variables $\mathcal{P}$, we can use the construction in Theorem \ref{maintheorem} to represent the final diffeomorphism between $\Psi(\mathcal{M})$ and the feature space $\mathcal{F}(\mathcal{N})$.  In the next section we demonstrate the IDM on two product manifolds, namely the annulus and the torus.  We will also attempt to apply the IDM to manifolds which are not product manifolds and report the empirical results.


\subsection{Examples}\label{examples}

In this section we will demonstrate how the iterated diffusion map is able to contract a manifold onto a lower-dimensional feature of interest.  All the examples use $M=250$ rescaled diffusion coordinates. We found the results to be robust down to around $M=100$ rescaled diffusion coordinates and no improvement above $M=250$.  In the examples below we adjusted the parameter $\tau \in (0,1)$, which defines the discretization of the geometric flow in Section \ref{IDM}, in order to achieve the desired feature in about four iterations of the diffusion map.  In principle, one would like to take $\tau$ as small a possible, however this requires many iterations that are computationally intensive.  Also, we have found that numerical errors can accumulate over large numbers of iterations, which we discuss in the Section \ref{conclusion}.  For a compact description of the numerical algorithm, see \ref{numerics}.

 \begin{figure}[h]
  \begin{center}
\includegraphics[trim={6cm 0 4.3cm 0},clip,width=.95\linewidth]{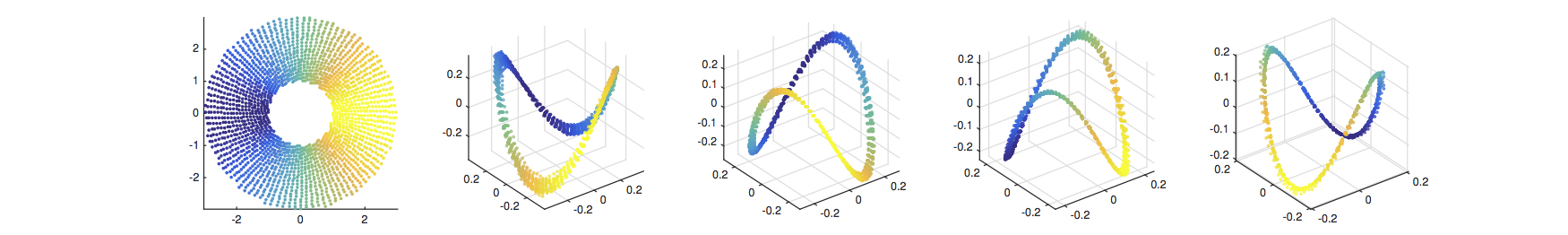}
 \includegraphics[trim={6cm 0 4.3cm 0},clip,width=.95\linewidth]{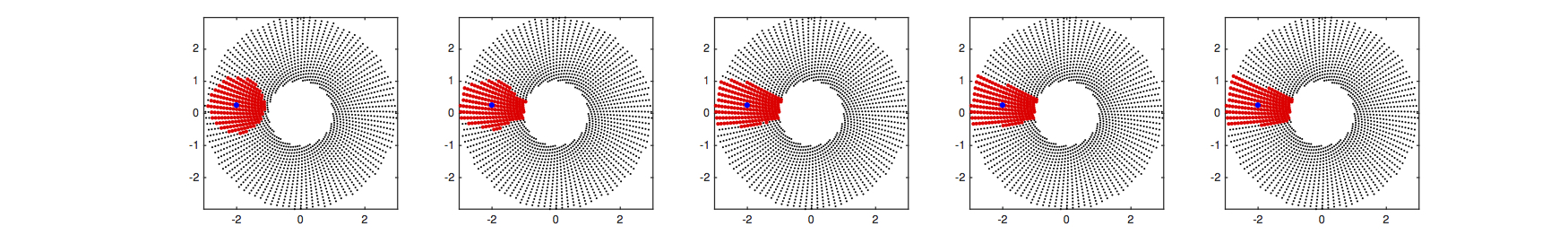}
\caption{\label{annulus2} Top: Original annulus data set colored according to the feature of interest (leftmost), followed by four iterations of the diffusion map using the local kernel defined in Section \ref{section4} with $\tau = 0.3$.  Each diffusion map shows the first three rescaled diffusion coordinates colored according to the feature of interest (the angle of the data point in the original annulus). However, 250 rescaled diffusion coordinates are maintained at each step.  Bottom: Original data set showing the 200 nearest neighbors (red) of the blue data point, where the neighbors are found in the corresponding iterated diffusion map embedding.}
\end{center}
\end{figure}

The first example is the annulus described in Section \ref{section1} which used $\tau = 0.65$.  The annulus is a product space $A = S^1 \times [1,3]$ and in Figure \ref{annulus1} we show the iterated diffusion map recovering the radial component.  If we parameterize the annulus with polar coordinates $(\theta,r) \in [0,2\pi) \times [1,3]$, then the feature of interest in Figure \ref{annulus1} was the coordinate $r$, which shows that the iterated diffusion map is able to change the topology of a manifold (both the dimension and the number of holes are changed in Figure \ref{annulus1}).  Notice that although both the source and target manifolds are less than three dimensional, the iterated diffusion map must move through a three-dimensional embedding in order to transition between these very different geometries.  Indeed, the first application of the diffusion map (with the local kernel described in Section \ref{IDM}) shown in Figure \ref{annulus1} transforms the geometry from an annulus to a cylinder.  Intuitively the cylinder introduces a new variable, height, to represent the feature of interest. This is shown by the coloring in Figure \ref{annulus1} which represents the radius of each point on the original annulus, and varies only with the height of the cylinder.  As the diffusion map is iterated, the geometry evolves as described in Section \ref{section4}, intuitively putting more emphasis on the direction (namely the height) which contains the radial information.  This is manifested as the circle component of the cylinder contracting until the data set becomes a line, thereby representing only the feature of interest as shown by the coloring.

 \begin{figure}[h]
  \begin{center}
\includegraphics[trim={6cm 0 4.5cm 0},clip,width=.95\linewidth]{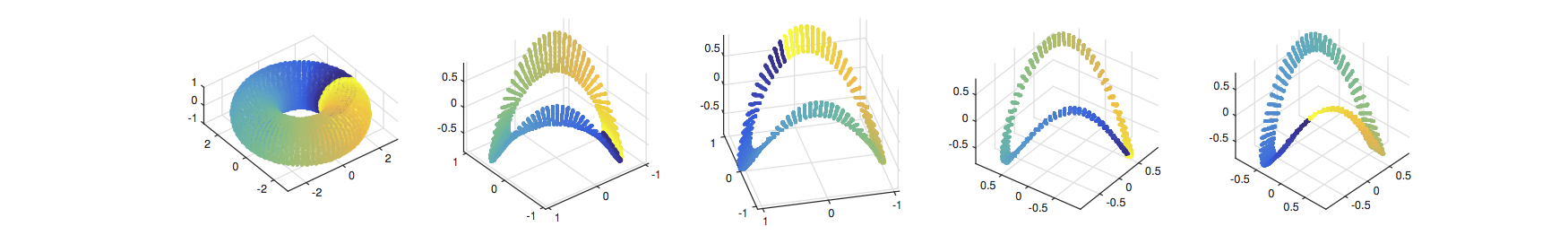}
 \includegraphics[trim={6cm 0 4.5cm 0},clip,width=.95\linewidth]{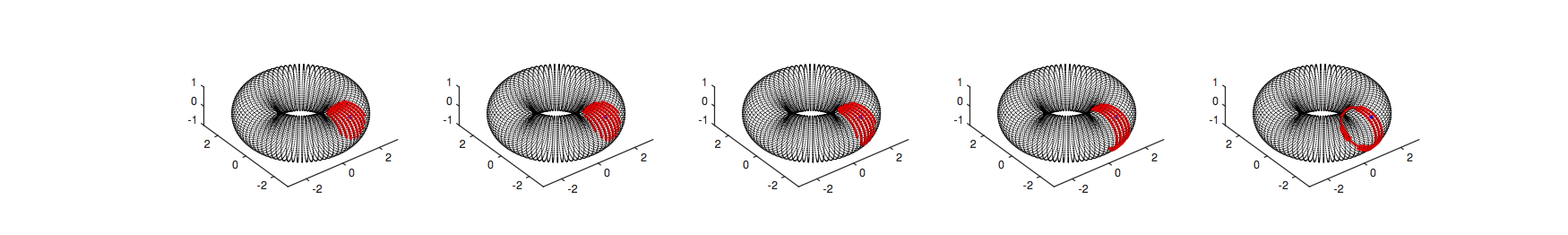}
 \includegraphics[trim={6cm 0 4.5cm 0},clip,width=.95\linewidth]{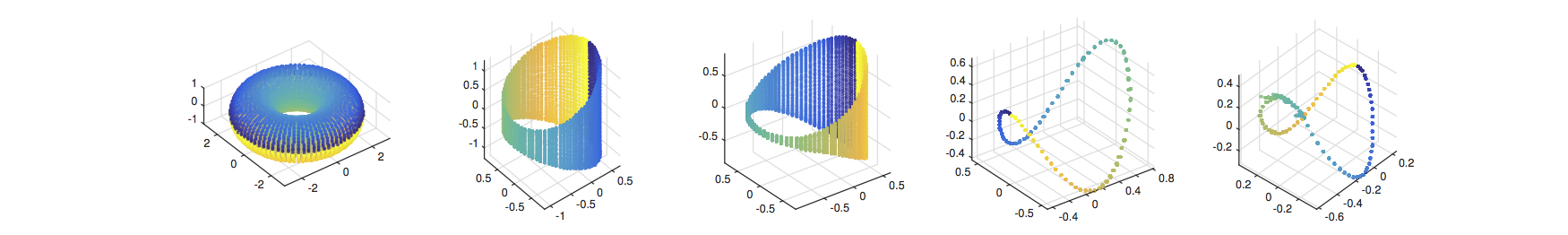}
 \includegraphics[trim={6cm 0 4.5cm 0},clip,width=.95\linewidth]{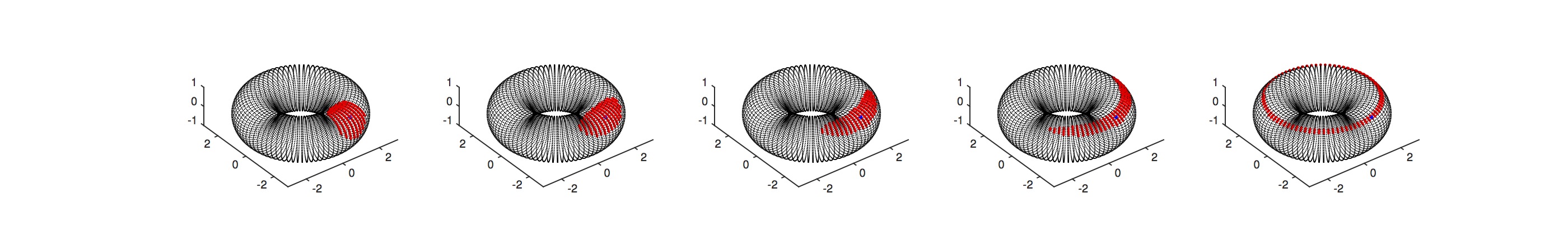}
\caption{\label{torusIDM} Top Row: Original data set colored according to the desired feature (leftmost) followed by four iterations of the IDM with the feature of interest given by $\mathcal{H}(x,y,z) = (\sin(\phi),\cos(\phi))^\top$ and $\tau = 0.4$.  Second Row: Original data set showing the 200 nearest neighbors (red) of the blue data point, where the neighbors are found in the corresponding iterated diffusion map space from the top row.  Third Row: Original data set colored according to the desired feature (leftmost) followed by four iterations of the IDM with the feature of interest given by $\mathcal{H}(x,y,z) = (\sin(\theta),\cos(\theta))^\top$ and $\tau = 0.65$.  Bottom Row: Original data set showing the 200 nearest neighbors (red) of the blue data point, where the neighbors are found in the corresponding iterated diffusion map space from the third row.}
\end{center}
\end{figure}

We now show that the IDM can also contract the annulus onto the other natural feature of interest, namely the angle.  However, note that the single parameter $\theta \in [0,2\pi)$ is not an embedding of the circle, since the periodic boundary conditions cannot be satisfied in $\mathbb{R}^1$.  Instead, to recover the circle from the annulus, the feature of interest is the two-dimensional feature $(\sin(\theta),\cos(\theta))^\top$, which is an embedding of the circle.  In Figure \ref{annulus2} we show the results of applying the iterated diffusion map to the annulus with the feature $\mathcal{H}(\theta,r) = (\sin(\theta),\cos(\theta))^\top$.  

Next we consider a simple example where the manifold is a torus $T^2 = S^1 \times S^1$ with intrinsic coordinates $(\theta,\phi) \in [0, 2\pi)^2$ with periodic boundary conditions.  Since the torus is a product of two circles, parameterized by $\theta$ and $\phi$, respectively, we can consider either of these circles as a lower dimension feature of interest.  For example, when the desired feature is the circle parameterized by $\theta$, the feature valued function is $\mathcal{H}(x,y,z) = (\sin(\theta),\cos(\theta))^\top$.  As shown in Figure \ref{torusIDM}, when the feature of interest on the torus is either of the circles in the product structure, the IDM evolves the manifold by contracting the irrelevant circle until only the feature of interest remains.  Notice that the IDM is able to destroy topological features such as holes in pursuit of the feature of interest.

 \begin{figure}[h]
  \begin{center}
\includegraphics[trim={6cm 0 4.5cm 0},clip,width=.95\linewidth]{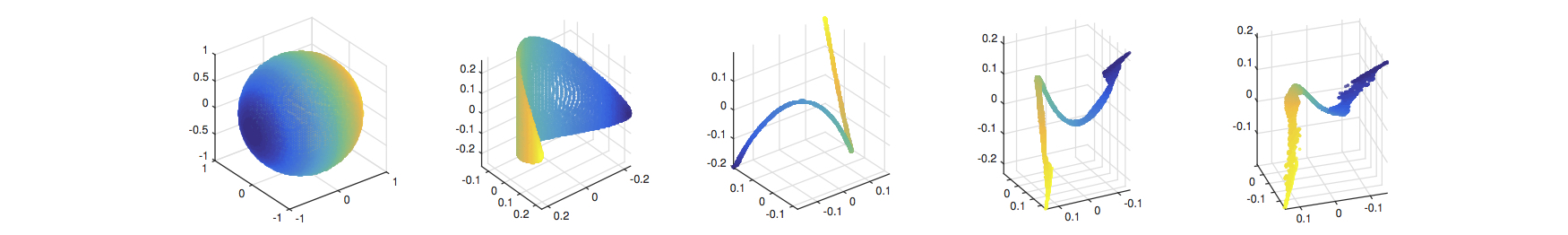}
 \includegraphics[trim={6cm 0 4.5cm 0},clip,width=.95\linewidth]{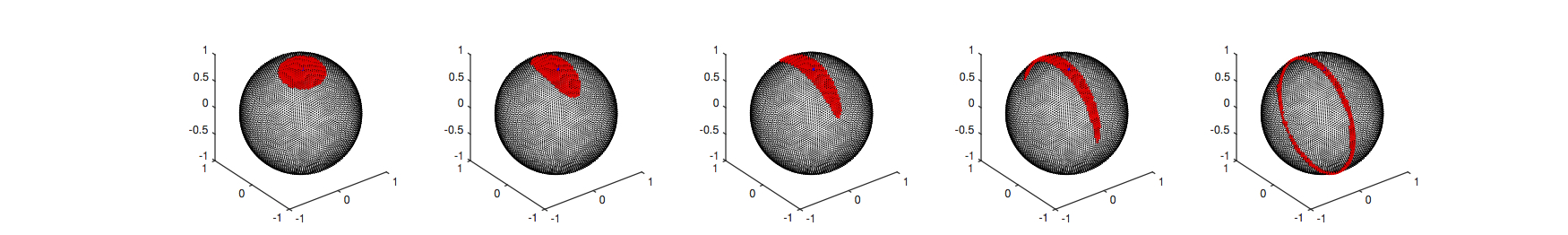}
 \includegraphics[trim={6cm 0 4.5cm 0},clip,width=.95\linewidth]{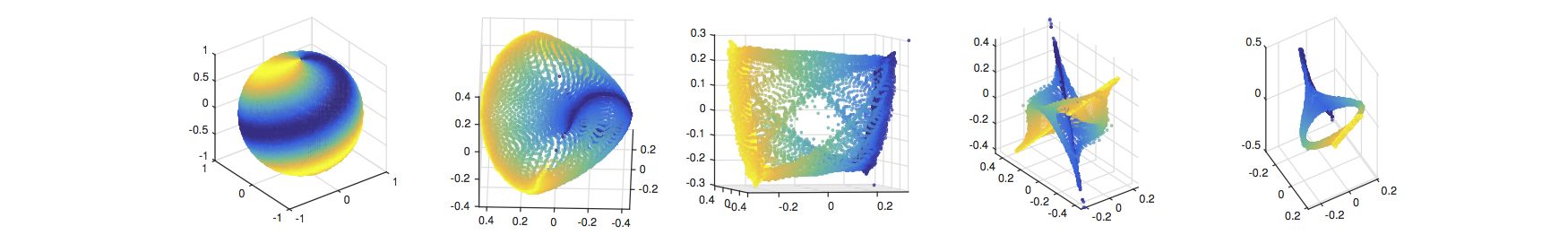}
 \includegraphics[trim={6cm 0 4.5cm 0},clip,width=.95\linewidth]{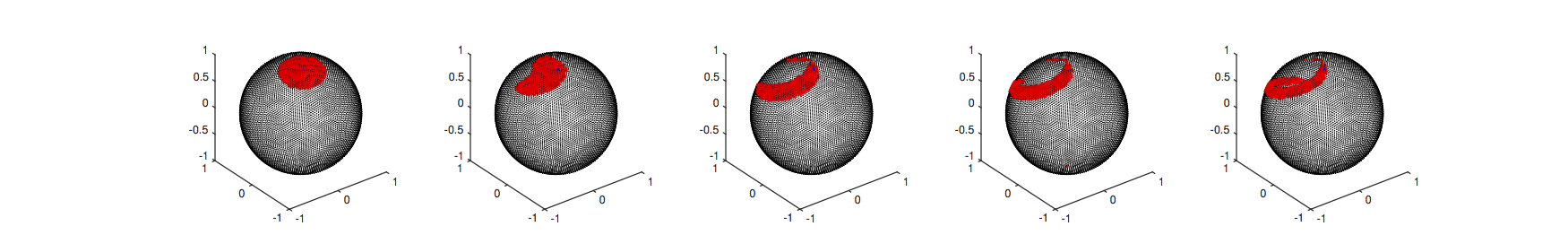}
\caption{\label{sphereIDM} Top: Original data set colored according to the desired feature (leftmost), followed by four iterations of the diffusion maps is the local kernel defined in Section \ref{section4} with $\tau = 0.7$.  Second Row: Original data set showing the 400 nearest neighbors (red) of the blue data point, where the neighbors are found in the corresponding iterated diffusion map space.  Third Row: Original data set colored according to a more complex desired feature (leftmost), followed by four iterations of the IDM with the feature of interest given by $\mathcal{H}(x,y,z) = \sin(\pi z/2 + \tan^{-1}(y/x))$ and $\tau = 0.6$.  Bottom Row: Original data set showing the 400 nearest neighbors (red) of the blue data point, where the neighbors are found in the corresponding iterated diffusion map space from the third row.}
\end{center}
\end{figure}

Finally, we consider a manifold that is not a product space, namely a 2-dimensional unit sphere, and we first choose the feature of interest to be simply the $x$-coordinate of the sphere and the results for this example are shown in Figure \ref{sphereIDM}.  We also consider a sphere with a more complex feature that twists up the sphere, as shown in the third row of Figure \ref{sphereIDM}.  While in these cases the geometric flow cannot be described by the simple product formula in \eqref{prodflow}, the flow still emphasizes the feature of interest and seems to contract the manifold onto a lower dimensional manifold that better represents the feature.

\section{Conclusion}\label{conclusion}

The above results show that for intrinsically low-dimensional data sets, an iterated diffusion map can help identify features that are hidden in the geometric structure of the data.  From a geometric point of view, the IDM approximates a geometric flow that stretches directions on the manifold that are locally correlated to the desired feature and contracts directions that are locally uncorrelated with the desired feature.  When the data manifold is a product of the feature manifold and the irrelevant variables, the geometric flow \eqref{flow} reduces to \eqref{prodflow}, which stably contracts the irrelevant variables to zero.  So for product manifolds, the IDM constructs the quotient map from the product manifold to the feature manifold.  For more general manifolds, this geometric flow appears empirically to converge to a lower-dimensional manifold that better represents the feature of interest.

Several key tools are necessary for the construction of the IDM.  First, as shown in Section \ref{section3}, one needs to be able to estimate the local derivatives of a nonlinear map between manifolds embedded in Euclidean space using only empirical data.  Second, the construction of Section \ref{section4} is necessary to form a local kernel that satisfies the requirements of Theorem \ref{maintheorem}.  Finally, the rescaled diffusion mapping of Section \ref{section2} is needed to give an isometric embedding of the new geometry introduced by the local kernel.  With these three pieces in place, it is possible to iterate the diffusion map in a way that approximates a geometric flow and emphasizes the variable of interest.

However, several important problems remain unsolved.  First, we found empirically that applying too many iterations of the diffusion map lead to apparent numerical instability.  We suspect that this problem arises from accumulated numerical error in the repeated eigensolves required to find the diffusion mappings.  Second, the exact criterion for the convergence of the iterated diffusion map to the desired feature requires a better theoretical understanding of the geometric flow of Section \ref{section4}, such as its attracting set and the stability of the equilibrium points.  Finally, the current algorithm requires the entire manifold to be well-sampled, which means that the data requirements depend on the dimension of the data set and not simply the feature set.  Intuitively, it may be possible to extend the IDM to points that lie in sparsely-sampled regions of the data set, as long as these regions are well-sampled in the feature space.  Ideally, this could reduce the data requirements to only depend on the dimensionality of the very low-dimensional feature set.  However, it is unclear how to extrapolate the IDM into these sparsely-sampled regions of data space, and currently the need to estimate the local derivatives requires fairly dense sampling everywhere on the data manifold.

\section*{Acknowledgments}
The research of J.H. is partially supported by the Office of Naval Research Grants N00014-11-1-0310, N00014-13-1-0797, MURI N00014-12-1-0912 and the National Science Foundation DMS-1317919. T. B. was supported under the ONR MURI grant N00014-12-1-0912.

\bibliographystyle{plain}
\bibliography{VBbib}

\appendix

\section{Numerical Algorithm for the Iterated Diffusion Map} \label{numerics}

Given a data set $\{x_i\}_{i=1}^N \subset \mathcal{M} \subset \mathbb{R}^m$ and a feature $\{y_i=\mathcal{H}(x_i)\}_{i=1}^N \subset \mathcal{N} \subset \mathbb{R}^n$, the algorithm of this section can be used to construct an embedding $\Psi(x_i)$ which emphasizes the feature of interest.  We assume that $x_i$ are points in $\mathbb{R}^m$ which are sampled on (or very close to) a $d$-dimensional manifold $\mathcal{M}$, and that the feature of interest $y_i = \mathcal{H}(x_i)$ lives on a manifold $\mathcal{N}$ which has dimension less than or equal to $d$.  The algorithm also produces a basis of eigenfunctions $\{\psi_j(x_i)\}_{j=1}^M$ that can be used to represent the map $\mathcal{H}$ and extend this mapping to out-of-sample data points $x$ with standard methods such as the Nystr\"om extension.  

The first part of the IDM is a generic algorithm for estimating the derivative $D\mathcal{H}(x_i)$ at each point.  The step-by-step algorithm is outlined in the first box below.  Optionally, if the embedding of $\mathcal{M}$ has high curvature or the data is noisy, a more robust method of tuning the local bandwidth parameter may be used by adding the following substeps to Step 4. and then replacing Step 5. with ``Set $\ell = \textup{argmin}(M)$ and $d(i) = d_{\textup{ave}}(\ell)$.'' 
  
  \begin{enumerate} 
  		\item[(e)] Form the $k\times m$ matrix $\tilde X$ of weighted vectors with $j$-th column $\tilde X_j = \sqrt{\frac{w_j}{D(\ell)}}(x_{I(j)} - x_i)$
		\item[(f)] Compute the singular values $\sigma_1,...,\sigma_m$ of $X$ and store them in $s(\ell,j) = \sigma_j$ for $j=1,...,m$
		\item[(g)] If $\ell > 1$ Compute the scaling law of each singular value $\alpha(\ell-1,j) = \frac{\log(s(\ell,j))-\log(s(\ell-1,j))}{\log(\epsilon(\ell))-\log(\epsilon(\ell-1))}$
		\item[(h)] If $\ell > 1$ Set $d_0 = \textup{floor}(d_1(\ell-1))$ and compute $d_2(\ell-1) = 2 \sum_{j=1}^{d_0} \alpha(\ell-1,j) + 2(d_1-d_0) \alpha(\ell-1,d_0+1)$
		\item[(i)]  If $\ell > 1$ set $d_{\textup{ave}}(\ell-1) = (d_1(\ell-1) + d_2(\ell-1))/2$
		\item[(j)] If $\ell > 1$ set $M(\ell-1) = \left|\frac{d_1(\ell-1)-d_2(\ell-1)}{d_{\textup{ave}}(\ell-1)}\right| + \left|\frac{\log(d_1(\ell))-\log(d_1(\ell-1))}{\log(\epsilon(\ell))-\log(\epsilon(\ell-1))}\right| + \left| \frac{\log(d_2(\ell))-\log(d_2(\ell-1))}{\log(\epsilon(\ell))-\log(\epsilon(\ell-1))} \right|$
		\end{enumerate}

\framebox{
 \begin{minipage}{.96\linewidth}
\vspace{0.2cm}
\large{Algorithm 1: Estimating $D\hat{\mathcal{H}}(x_i)$ with Auto-tuned Bandwidth}
\small

\vspace{5pt} {\bf Inputs:} Data sets $\{x_i\}_{i=1}^N \subset \mathbb{R}^m$ and $\{y_i\}_{i=1}^N \subset \mathbb{R}^n$.  Parameters: number of nearest neighbors, $k$, and number of discrete values of the bandwidth to consider, $L$.

\vspace{5pt} {\bf Outputs:} At each point $x_i$ the algorithm returns the local estimates of the derivative $D\hat{\mathcal{H}}(x_i)$ and the intrinsic dimension $d(i)$ and the sampling density $q(i) = q(x_i)$.

\vspace{5pt} For each $i = 1,...,N$ 
	\begin{enumerate}
	\item Find the $k$-nearest neighbors of $x_i$ in ${\mathbb R}^{m}$, let their indices be $I(j)$ (where $I(1) = i$) for $j=1,...,k$ ordered by increasing distance and let $d(j) = ||x_i-x_{I(j)}||$.  In Section \ref{examples} we used $k=500$. \vspace{2pt}
	\item Tune the local bandwidth $\epsilon$ using steps (c)-()
	\item Define $\epsilon_{\textup{min}} = d(2)/(2\log \epsilon_{\textup{MACH}})$ and $\epsilon_{\textup{max}} = 10 d(k)$
	\item For $\ell=1,...,L$ 
		\begin{enumerate} 
		\item Let $\epsilon(\ell) = \exp\left(\log(\epsilon_{\textup{min}}) + (\ell/L)(\log(\epsilon_{\textup{max}})-\log(\epsilon_{\textup{min}})) \right)$
		\item Compute the weights $w_j = \exp\left(-d(j)^2/(2\epsilon(\ell))\right)$
		\item Compute the sum $D(\ell) = \sum_{j=1}^k w_j$
		\item If $\ell > 1$, compute $d_1(\ell-1) = 2\frac{\log(D(\ell))-\log(D(\ell-1))}{\log(\epsilon(\ell))-\log(\epsilon(\ell-1))}$
		\end{enumerate}
	\item Set $\ell = \textup{argmax}(d_1)$ and $d(i) = d_1(\ell)$
	\item Set $\epsilon = \epsilon(\ell+1)$
	\item Compute the weights $w_j = \exp\left(-d(j)^2/(2\epsilon)\right)$
	\item Compute the sum $D = \sum_{j=1}^k w_j$
	\item Set $q(i) = \frac{(2\pi\epsilon)^{d(i)/2}}{N}D$
	\item Form the $k\times m$ matrix $X$ of weighted vectors with $j$-th column $X_j = \sqrt{\frac{w_j}{D}}(x_{I(j)} - x_i)$
	\item Form the $k\times n$ matrix $Y$ of weighted vectors with $j$-th column $Y_j = \sqrt{\frac{w_j}{D}}(y_{I(j)} - y_i)$
	\item Compute the $m\times n$ matrix $D\hat{\mathcal{H}}(i)$ using the linear least squares regression $D\hat{\mathcal{H}}(i) = (X^\top X)^{-1} X^{\top} Y$
  \end{enumerate}
  \end{minipage}
  } 
		
\vspace{15pt} To compute the Iterated Diffusion Map (IDM), we will iteratively construct $x_i^\ell$, where $x_i^1 = x_i$ and $\ell$ runs up to the desired number of iterations, $t$.  Determining a good stopping criterion is a difficult problem.  If the goal is to estimate the map $\mathcal{H}$ using the method of Section \ref{productman}, then a promising approach is to use cross-validation.  This approach would first compute the rescaled diffusion coordinates of the feature $y_i$ and then attempt a linear regression from $x_i^\ell$ to these diffusion coordinates and iterate until the residual ceases to decrease.  Another significant issue is that the theory of \cite{LK} has not yet been extended to use the variable bandwidth kernels of \cite{BH14VB}.  So even though we have an estimate of the optimal bandwidth at each point, we can only apply Theorem \ref{thm1} with a fixed global bandwidth.  In our examples we found the best choice of global bandwidth was a simple average of the squared distances to the $k=32$ nearest neighbors of each point, averaged over the whole data set.  One reason this ad hoc bandwidth is required is due to the large variations in the dimension that occur as the diffusion map iterates, see for example Figure \ref{annulus1} where some parts of the annulus contract to a line before others.  These variations of the dimension also require us to use a locally rescaled diffusion mapping.  Notice that Step 8 (a)-(d) are the standard diffusion map algorithm using the local kernel defined by $C_{\mathcal{H}}$ and the associated distances $d_{\mathcal{H}}$.  However, to normalize the eigenfunctions in Step 9, we use the locally estimated sampling density $q(i)$.  Also, to form the rescaled diffusion mapping in Step 11, we use the locally estimated dimension $d(i)$.  We found that when a global kernel density estimate and a globally estimated dimension were used, the distances were scaled very differently in different parts of the data set and this distortion led to numerical problems after several iterations.

 \vspace{15pt} \framebox{
 \begin{minipage}{.96\linewidth}
\vspace{0.2cm}
\large{Algorithm 2: The Iterated Diffusion Map (IDM)}
\small

\vspace{5pt} {\bf Inputs:} Data sets $\{x_i\}_{i=1}^N \subset \mathbb{R}^m$ and $\{y_i\}_{i=1}^N \subset \mathbb{R}^n$.  Parameters: number of nearest neighbors, $k$, number of discrete values of the bandwidth to consider, $L$, number of nearest neighbors to use to estimate the global bandwidth, $k_2$, number of eigenfunctions to use in the diffusion map, $M$, geometric flow discretization parameter, $\tau$, and number of iterations, $t$.

\vspace{5pt} {\bf Outputs:} The $M$-dimensional IDM embedding $x_i^t = \Psi(x_i) = \Phi^{(t)} \circ \cdots \circ \Phi^{(1)}(x_i)$ which represent the feature $y_i$.

\vspace{5pt} Set $x_i^1 = x_i$

For each $\ell = 1,...,t$ 
	\begin{enumerate}
	\item Use Algorithm 1, with inputs $\{x_i^\ell\}$ and $\{y_i\}$ to estimate $D\hat{\mathcal{H}}(x_i^{\ell})$, the local dimension $d(i)=d(x_i^{\ell})$ and density $q(i) =q(x_i^{\ell})$
	\item Let $I(i,j)$ be the index of the $j$-th nearest neighbor of $x_i^{\ell}$ and let $d(i,j) = ||x_{I(i,j)}^{\ell}-x_i^{\ell}||$
	\item Define the distance $d_{\mathcal{H}}(i,j) = (1-\tau)d(i,j) + \tau ||D\hat{\mathcal{M}}(x_i)(x_{I(i,j)}^{\ell}-x_i^{\ell})||$ with respect to $C_{\mathcal{H}}(x_i)$ from \eqref{Cdel}
	\item Use the ad hoc global bandwidth estimate $\epsilon = \frac{1}{N k_2} \sum_{i=1,j=1}^{i=N,j=k_2} d_{\mathcal{H}}(i,j)$
	\item Build the local kernel $J(i,j) = \exp\left(-d_{\mathcal{H}}(i,j)^2/2\epsilon \right)$
	\item Build a sparse $N\times N$ matrix $\tilde J$ with $\tilde J_{i,I(i,j)} = J(i,j)$
	\item Symmetrize $\hat J = (\tilde J + \tilde J^\top)/2$
	\item Apply the standard diffusion maps normalizations as in \cite{diffusion,LK}
	\begin{enumerate}
		\item Right Normalization: Set $D_i = \sum_j \hat J_{ij}$ and $K = \hat J_{ij}/(D_i D_j)$
		\item Left Normalization: Set $\hat D_i = \left(\sum_j K_{ij} \right)^{1/2}$ and $\hat K = K_{ij}/(\hat D_i \hat D_j)$
		\item Compute the $M+1$ largest eigenvalues $\xi_{r}$ and associated eigenvectors $\tilde \varphi_{r}$ of $\hat K$ for $r=0,...,M$
		\item Define $\hat \varphi_{r}(x_i^{\ell}) = \tilde \varphi_{\ell}(x_i^{\ell})/\hat D_i$
	\end{enumerate}
	\item Normalize the eigenvectors with respect to the sampling density $\varphi_{r} = \hat \varphi_{r}/\sqrt{\frac{1}{N}\sum_{i=1}^N \hat \varphi_{r}(x_i^{\ell})^2/q(i)}$
	\item Set $s = 10\epsilon$ and $\lambda_{r} = \log(\xi)/\epsilon$
	\item Define the rescaled diffusion mapping $x_i^{\ell+1} = \Phi_s^{(\ell)}(x_i^\ell) = (2\pi)^{d(i)/4}(4s)^{d(i)/4 + 1/2} (e^{\lambda_{1}s}\varphi_1(x_i^\ell),...,e^{\lambda_{M}s}\varphi_M(x_i^\ell) )^\top \in \mathbb{R}^M$
  \end{enumerate}
  \end{minipage}
  }

\end{document}